\documentclass[12pt]{amsart}

\usepackage{amsfonts}
\usepackage{wasysym}
\usepackage{amsmath}
\usepackage{amssymb}
\usepackage{color}
\usepackage{graphicx}
\usepackage{xspace}
\usepackage{axodraw4j}
 \font \eightrm=cmr8
 
\input xy
\xyoption{all}

\newcommand{\nc}{\newcommand}

\nc{\smop}[1]{\mathop{\hbox {\eightrm #1} }\nolimits}

\newtheorem{thm}{Theorem}
\newtheorem{exam}{Example}

\newtheorem{cor}[thm]{Corollary}

\newtheorem{lem}[thm]{Lemma}
\newtheorem{prop}[thm]{Proposition}

\newcommand*{\id}{{{id}}}
\newcommand*{\un}{{\mathbf 1}}

\def\shuff#1#2{\mathbin{
      \hbox{\vbox{\hbox{\vrule \hskip#2 \vrule height#1 width 0pt}\hrule}\vbox{\hbox{\vrule \hskip#2 \vrule height#1 width 0pt\vrule }\hrule}}}}
\def\shuffl{{\mathchoice{\shuff{5pt}{3.5pt}}{\shuff{5pt}{3.5pt}}{\shuff{3pt}{2.6pt}}{\shuff{3pt}{2.6pt}}}}
\def\shuffle{{\, \shuffl \,}}

\def\feedbackA{\scalebox{0.8}{
\fcolorbox{white}{white}{
  \begin{picture}(304,107) (33,-54)
    \SetWidth{1.0}
    \SetColor{Black}
    \Line[arrow,arrowpos=0.5,arrowlength=5,arrowwidth=2,arrowinset=0.2](64,26)(122,27)
    \Line[arrow,arrowpos=0.5,arrowlength=5,arrowwidth=2,arrowinset=0.2](138,26)(192,27)
    \Line[arrow,arrowpos=0.5,arrowlength=5,arrowwidth=2,arrowinset=0.2](224,27)(288,27)
    \Vertex(288,27){3}
    \Line[arrow,arrowpos=0.5,arrowlength=5,arrowwidth=2,arrowinset=0.2](288,27)(336,27)
    \COval(128,27)(10,10)(0){Black}{White}
    \Line[arrow,arrowpos=0.5,arrowlength=5,arrowwidth=2,arrowinset=0.2](288,27)(288,-37)
    \Line[arrow,arrowpos=0.5,arrowlength=5,arrowwidth=2,arrowinset=0.2](288,-37)(224,-37)
    \Line[arrow,arrowpos=0.5,arrowlength=5,arrowwidth=2,arrowinset=0.2,flip](128,-37)(192,-37)
    \Line[arrow,arrowpos=0.5,arrowlength=5,arrowwidth=2,arrowinset=0.2,flip](129,16)(128,-37)
    \Text(240,43)[]{\Large{\Black{$y$}}}
    \Text(48,27)[]{\Large{\Black{$\beta$}}}
    \Text(128,27)[]{\Large{\Black{$+$}}}
    \Text(176,43)[]{\Large{\Black{$u$}}}
    \CBox(192,11)(224,43){Black}{White}
    \Text(208,27)[]{\Large{\Black{$F_c$}}}
    \Line(224,-53)(224,-21)
    \Line(224,-21)(192,-37)
    \Text(208,-37)[l]{\Large{\Black{$\alpha$}}}
    \Line(224,-53)(192,-37)
  \end{picture}
}
}}

\def\feedbackcatd{\scalebox{0.8}{
\fcolorbox{white}{white}{
  \begin{picture}(304,108) (33,-54)
    \SetWidth{1.0}
    \SetColor{Black}
    \Line[arrow,arrowpos=0.5,arrowlength=5,arrowwidth=2,arrowinset=0.2](64,27)(122,28)
    \Line[arrow,arrowpos=0.5,arrowlength=5,arrowwidth=2,arrowinset=0.2](138,27)(192,28)
    \Line[arrow,arrowpos=0.5,arrowlength=5,arrowwidth=2,arrowinset=0.2](224,28)(288,28)
    \Vertex(288,28){3}
    \Line[arrow,arrowpos=0.5,arrowlength=5,arrowwidth=2,arrowinset=0.2](288,28)(336,28)
    \COval(128,28)(10,10)(0){Black}{White}
    \Line[arrow,arrowpos=0.5,arrowlength=5,arrowwidth=2,arrowinset=0.2](288,28)(288,-36)
    \Line[arrow,arrowpos=0.5,arrowlength=5,arrowwidth=2,arrowinset=0.2](288,-36)(224,-36)
    \Line[arrow,arrowpos=0.5,arrowlength=5,arrowwidth=2,arrowinset=0.2,flip](128,-36)(192,-36)
    \Line[arrow,arrowpos=0.5,arrowlength=5,arrowwidth=2,arrowinset=0.2,flip](128,17)(127,-37)
    \Text(240,44)[]{\Large{\Black{$y$}}}
    \Text(48,28)[]{\Large{\Black{$v$}}}
    \Text(128,28)[]{\Large{\Black{$+$}}}
    \Text(176,44)[]{\Large{\Black{$u$}}}
    \CBox(192,12)(224,44){Black}{White}
    \Text(208,28)[]{\Large{\Black{$F_c$}}}
    \CBox(191,-53)(223,-21){Black}{White}
    \Text(207,-37)[]{\Large{\Black{$F_d$}}}
  \end{picture}
}
}}

\def\re{{\mathbb R}} 

\def\allseries{\mbox{$\re\langle\langle X \rangle\rangle$}}

\def\gsc{c} 

\begin{document}

\title[Abel equation and output feedback]{Center problem, Abel equation and the \\ Fa\`a di Bruno Hopf algebra for output feedback}

\vspace{1cm}

\author{Kurusch Ebrahimi-Fard}
\address{ICMAT,
		C/Nicol\'as Cabrera, no.~13-15, 28049 Madrid, Spain {\tiny{(on leave from UHA, Mulhouse, France)}}}
         \email{kurusch@icmat.es, kurusch.ebrahimi-fard@uha.fr}
         \urladdr{www.icmat.es/kurusch}

\author{W.~Steven Gray}
\address{Old Dominion University,
		Norfolk, Virginia 23529 USA}
\email{sgray@odu.edu}
\urladdr{www.ece.odu.edu/$\sim$gray}

\date{\today}

\begin{abstract}
A combinatorial interpretation is given of Devlin's word problem underlying the classical center-focus problem of Poincar\'{e} for non-autonomous differential equations.
It turns out that the canonical polynomials of Devlin are from the point of view of connected graded Hopf algebras intimately related to the graded components of a Hopf algebra antipode applied to the formal power series of Ferfera.
The link is made by passing through control theory since the Abel equation, which describes a center, is equivalent to an output feedback equation, and the Hopf algebra of output feedback is derived from the {\em composition} of iterated integrals rather than just the products of iterated integrals, which yields the shuffle algebra. This means that the primary algebraic structure at play in Devlin's approach is actually not the shuffle algebra, but a Fa\`{a} di Bruno type Hopf algebra, which is defined in terms of the shuffle product but is a distinct algebraic structure.
\end{abstract}


\maketitle
\tableofcontents


\section{Introduction}
\label{sect:intro}

The classical center-focus problem first studied by Poincar\'{e} considers a system of equations
\begin{equation}
\label{eq:Poincare-problem}
	\frac{dx}{dt}=X(x,y),\;\;\; \frac{dy}{dt}=Y(x,y),
\end{equation}
where $X,Y$ are polynomial vectors
fields with a linear part of center or focus type and homogeneous nonlinearities \cite{Poincare_1892}. The equilibrium at the origin is a center if it is contained in an open neighborhood $U$ having no other equilibria, and every trajectory of (\ref{eq:Poincare-problem}) in $U$ is closed. An alternative way to characterize a center is to first put (\ref{eq:Poincare-problem}) into polar coordinates and then pass to the Abel equation 
\begin{equation}
\label{Devlineq}
		\dot{z}(t) = \alpha(t) z^3(t) + \beta(t) z^2(t)	+\gamma(t) z(t)
\end{equation}
via the Cherkas transformation \cite{Cherkas_76}.
This amounts to a re-parametrization so that the independent variable $t$ plays the role of the
angle in the polar coordinate framework \cite{Alwash-LLoyd_87,Lloyd_82}. In this setting,
the origin $z=0$ is a center if it is contained in a neighborhood $\tilde{U}$ such that every solution of
(\ref{Devlineq}) initialized at some $z(0)\in \tilde{U}$ is periodic. Under perturbation of the parameter functions $\alpha$, $\beta$ and $\gamma$, a given periodic solution, including the trivial solution $z=0$, can bifurcate into $\mu\geq 1$  distinct limit cycles. Hilbert's 16th problem asks for the maximum number of limit cycles, $\mu_{\rm max}$.
Not surprisingly, the literature on this problem is extensive (see \cite{Ilyashenko_02,Ilyashenko-Yakovenko_95} for an overview).  The main interest here, however, is in an algebraic approach first proposed by Devlin in 1989 \cite{Devlin_1989,Devlin_1991}, which is based on the work of Alwash and Lloyd \cite{Alwash-LLoyd_87,Lloyd_82}. The latter showed that to determine when $\mu>1$, there is no loss of generality in assuming $\gamma=0$. Specifically, starting from
\begin{equation}
\label{eq:Devlineq-without-gamma}
		\dot{z}(t) = \alpha(t) z^3(t) + \beta(t) z^2(t),
\end{equation}
the method allows one to determine an upper bound on $\mu_{\rm max}$ for various classes of coefficient functions $\alpha$ and $\beta$. (For more recent developments
along these lines, see \cite{Briskin-etal_10,Brudnyi_10}.) The general goal of this paper is to describe the main underlying combinatorial structure behind Devlin's method. This has become
evident only recently due to the application of combinatorial Hopf algebras in feedback control theory \cite{DuffautEFG_2014,Foissy_13,Gray-Duffaut_Espinosa_SCL11,Gray-Duffaut_Espinosa_FdB14,Gray-et-al_SCL14}.

The basic idea behind Devlin's approach is to identify an underlying word problem associated with the solution of (\ref{eq:Devlineq-without-gamma}).
For sufficiently small $r > 0$ in a neighborhood of the origin, let $\Phi(t;0;r;u)$ be the solution $z(t)$ with initial condition $z(0)=r$ and parameter
functions $u=\{\alpha, \beta \}$. In which case, it is possible to write the solution formally as the power series
\begin{equation}
\label{eq:fomal-solution-Abel-equation}
	\Phi(t;0;r;u)=\sum_{n=1}^\infty a_n(t) r^n,
\end{equation}
where the $a_n$ are analytic functions satisfying $a_1(0)=1$ and $a_n(0)=0$ for $n>1$.
For a fixed $\omega>0$, the first return map is defined to be $P(u)(r)=\Phi(\omega;0;r;u)$, and
therefore,
\begin{equation}\label{eq:first-return-map}
	P(u)(r)=\sum_{n=1}^\infty a_n(\omega)r^n.
\end{equation}
Any solution satisfying $z(\omega)=z(0)=r$
has the property that
\begin{displaymath}
	q(r):=P(u)(r)-r=(a_1(\omega)-1)r+\sum_{n=2}^\infty a_n(\omega)r^n=0.
\end{displaymath}
The function $q$ has a zero of multiplicity $\mu>1$ when $a_1(\omega)=1$, $a_n(\omega)=0$ for $n=2,3,\ldots,\mu-1$, and $a_{\mu}(\omega)\neq 0$.
If $\mu=\infty$ then $z=0$ is a center.
It is known that $\mu$ is equivalent to the number of limit cycles the periodic solution $z=0$ can exhibit for sufficiently small perturbations of $\alpha$ and $\beta$. The main idea therefore is to set up a system of equations in terms of the functions $a_n$ in order to bound $\mu$ as a function of $\alpha$ and $\beta$. To do this, Devlin expresses the solution $z(t)$ in terms of a {\em Chen-Fliess functional expansion} or {\em Fliess operator}, that is, a weighted sum of iterated integrals of $\alpha$ and $\beta$ so that each function $a_n$ has an associated generating polynomial of the same name in two noncommuting letters, $x_0$, corresponding to $\alpha$, and $x_1$, corresponding to $\beta$. These canonical polynomials, which Devlin derives from the underlying shuffle algebra over the set of words $X^\ast$ from the alphabet $X=\{x_0,x_1\}$, are central to the analysis. Of special significance is the fact that they are related by the linear recursion
\begin{equation} \label{eq:Devlin-recursion}
	a_n=(n-1)a_{n-1}x_1+(n-2)a_{n-2}x_0,\;\;n\geq 2
\end{equation}
with $a_1=1$ and $a_2=x_1$. For example,
\begin{subequations}
\label{eq:first-few-Devlin-polynomials}
\begin{align}
a_3&=2x_1^2+x_0 \\
a_4&=6x_1^3 +3x_0x_1+2x_1x_0 \label{eq:Devlin-polynomial-a4}\\
a_5&=24x_1^4+12x_0x_1^2+8x_1x_0x_1+6x_1^2x_0+3x_0^2 \\
a_6&=120x_1^5+60x_0x_1^3+40x_1x_0x_1^2+30x_1^2x_0x_1+15x_0^2x_1+24x_1^3x_0+12x_0x_1x_0+8x_1x_0^2.
\end{align}
\end{subequations}
Devlin shows that this recursion yields a remarkably simple
expression for the coefficient of any word $\eta=x_{i_1}\cdots x_{i_k}\in X^\ast$ of $a_n$:
\begin{equation} \label{eq:Devlin-coefficients}
\langle a_n,\eta\rangle=
\left\{
\begin{array}{ccl}
0&:& {\rm deg}(\eta)\neq n \\
1 &:& {\rm deg}(\eta)=n,\;k=0,1 \\
\displaystyle{\prod_{j=1}^{k-1} {\rm deg}(x_{i_1}\cdots x_{i_j})} &:& {\rm deg}(\eta)=n,\;k\geq 2,
\end{array}
\right.
\end{equation}
where the length of the word $\eta$ is denoted $|\eta|=k$, $|\eta|_{x_i}$ is the number of times the letter $x_i$ appears in
$\eta$, and the degree of $\eta$ is defined to be ${\rm deg}(\eta)=2|\eta|_{x_0}+|\eta|_{x_1}+1$. For example,
$\langle a_5,x_1x_0x_1\rangle=\deg(x_1)\deg(x_1x_0)=2\cdot 4$.

The specific goal here is to describe the precise combinatorial nature of Devlin's polynomials $a_n$ and the origin of their linear recursion \eqref{eq:Devlin-recursion}. It turns out that from the point of view of connected graded Hopf algebras, the $a_n$ are intimately related to the graded components of a Hopf algebra antipode applied to the formal power series $-c$, where
\begin{equation}
\label{FFsystem}
	c=\sum_{k = 0}^\infty k!\, x_1^k
\end{equation}
is the so called Ferfera series \cite{Ferfera_79,Ferfera_80}. Somewhat surprisingly, this link is made by passing through control theory. This is ultimately due to the fact that the Abel equation (\ref{Devlineq}) is equivalent to an output feedback equation in control theory, and the Hopf algebra of output feedback is derived from the {\em composition} of iterated integrals rather than just the products of iterated integrals, which yields the shuffle algebra \cite{DuffautEFG_2014,Gray-Duffaut_Espinosa_SCL11,Gray-et-al_SCL14}. This means that the primary algebraic structure at play in Devlin's approach is actually not the shuffle algebra, but a Fa\`{a} di Bruno type Hopf algebra, which is defined in terms of the shuffle product but is a distinct algebraic structure. So after some preliminaries presented in Section~\ref{sect:prelim}, this description of the $a_n$ in terms of an antipode is developed in Section~\ref{sect:Devlin-polynomials}. Now it is a standard theorem that the antipode of every connected graded Hopf algebra can be computed recursively \cite{Figueroa-Gracia-Bondia_05,manchon2}. This fact was exploited, for example, in the authors' application of the output feedback Hopf algebra to compute the feedback product, a device used in control theory to describe the feedback interconnection of two Fliess operators \cite{DuffautEFG_2014,Gray-et-al_MTNS14}. It is shown here, however, that this specific Hopf algebra has another type
of antipode recursion, one that is related to derivations on the shuffle algebra and distinct from the standard recursion. This result is developed in Section~\ref{sect:new-antipode-recursion}. Then in Section~\ref{sect:devlin-recursion} it is shown that Devlin's recursion (\ref{eq:Devlin-recursion}) is a specific example of
this new type of antipode recursion.

\medskip

{\bf{Acknowledgements}}
The first author was supported by Ram\'on y Cajal research grant RYC-2010-06995 from the Spanish government
and acknowledges support from the Spanish government under project MTM2013-46553-C3-2-P.
This research was also supported by a grant from the BBVA Foundation. The authors wish to thank Mr.~Lance
Berlin for developing the Mathematica software used to check some of the results in this manuscript and for
pointing out Corollary~\ref{co:Berlin-identity}.



\section{Preliminaries}
\label{sect:prelim}


\subsection{Combinatorial Hopf algebras}

In this section, a few basics on connected graded Hopf algebras are collected to fix the notation. For details, in particular on Hopf algebras of a combinatorial nature, the reader is referred to \cite{cartier1,ck,Figueroa-Gracia-Bondia_05,GBFV,manchon2}, as well as the standard reference~\cite{Sweedler_69}.

In general, $k$ denotes the ground field of characteristic zero over which all algebraic structures are defined. A counital {\it{coalgebra}} over $k$ consist of a $k$-vector space $C$ on which a coassociative coproduct $\Delta: C \to C \otimes C$ exists, that is, $(\Delta \otimes \id)\circ \Delta=(\id \otimes \Delta)\circ \Delta$ together with the counit map $\varepsilon: C \to k$. A $k$-linear coderivation $D: C \to C$ satisfies the coalgebraic analogue of the Leibniz rule, $\Delta \circ D=(D \otimes \id + \id \otimes D) \circ \Delta:C \to C \otimes C$. A {\it{bialgebra}} $B$ is both a unital algebra and a counital coalgebra together with compatibility relations, such as both the algebra product and unit map are coalgebra morphisms \cite{Sweedler_69}. The product and unit of $B$ are denoted by $m$ and $\mathbf{1}$, respectively.  A bialgebra is called {\it{graded}} if there are $k$-vector subspaces $B^{(n)}$, $n \geq 0$ such that $B= \bigoplus_{n \geq 0} B^{(n)}$, $m(B^{(p)} \otimes B^{(q)}) \subseteq B^{(p+q)}$, and $\Delta B^{(n)} \subseteq \bigoplus_{p+q=n} B^{(p)}\otimes B^{(q)}.$ Elements $x \in B^{(n)}$ are given a degree ${\rm deg}(x)=n$. Moreover, $B$ is called {\it{connected}} if $B^{(0)} = k\mathbf{1}$. Define $B^+=\bigoplus_{n > 0} B^{(n)}$. For any $x \in B^{(n)}$ the coproduct is of the form
\begin{equation}
\label{coprod}
	\Delta x = x \otimes \mathbf{1} + \mathbf{1} \otimes x + \Delta' x \in \bigoplus_{p+q=n} B^{(p)} \otimes B^{(q)},
\end{equation}
where $\Delta' x := \Delta x -  x \otimes \mathbf{1} - \mathbf{1} \otimes x \in B^+ \otimes B^+$ is the {\em reduced} coproduct. By definition an element $x \in B$ is {\it{primitive}} if $\Delta' x = 0.$ It is common to use variants of Sweedler's symbolic notation for the (reduced) coproduct, such as $\Delta x =\sum x_{(1)}\otimes x_{(2)}$ ($\Delta' x ={\sum}' x_{(1)} \otimes x_{(2)}$) or $\Delta x= x^\prime\otimes x^{\prime\prime}$. The action of the grading operator $Y : B \to B$, which is both a derivation and coderivation, is given by $Y h := n h$, where $h$ is a homogeneous element of degree $n$ in $B$. Suppose $A$ is a $k$-algebra with product $m_A$ and unit map $e_A: k \to A$, e.g., $A=k$ or $A=B$. The vector space  $L(B, A)$ of linear maps from the bialgebra $B$ to $A$ together with the convolution product $\Phi \star \Psi := m_{A} \circ (\Phi \otimes \Psi) \circ \Delta : B \to A$, where $\Phi,\Psi \in L(B,A)$, is an associative algebra with unit $\iota := e_{A} \circ \varepsilon$.

A {\it{Hopf algebra}} $H$ is a bialgebra together with a particular $k$-linear map called an {\it{antipode}} $S: H \to H$ which satisfies the Hopf algebra axioms~\cite{Sweedler_69}. When $A=H$, the antipode $S \in L(H,H)$ is characterized as the inverse of the identity map with respect to the convolution product
\begin{equation}
\label{antipode1}
    S  \star \id = \id \star S = e \circ \varepsilon.
\end{equation}
The next theorem is standard \cite{Figueroa-Gracia-Bondia_05,manchon2} and states that any connected graded bialgebra $B$ automatically is a {\sl{connected graded Hopf algebra}\/}.

\begin{thm}\label{th:antipode-induction}
On any connected graded bialgebra $(B,m,\Delta)$ an antipode $S$ can be defined inductively by $S\un:=\un$ and the following two equivalent recursions on $B^+=\rm{Ker}(\varepsilon)$, which follow directly from \eqref{antipode1},
\begin{equation*}
\label{antipode2}
	S = - \id - m \circ (S \otimes \id) \circ \Delta' = - \id - m \circ (\id \otimes S) \circ \Delta',
\end{equation*}
where $\Delta^\prime$ is the reduced coproduct defined in \eqref{coprod}.
\end{thm}

Let $H=\bigoplus_{n \ge 0} H^{(n)}$ be a connected graded Hopf algebra. Suppose $A$ is a commutative unital algebra. The subset $g_0 \subset  L( H, A)$ of linear maps $\alpha$ that send the unit of $H$ to zero, $\alpha(\mathbf{1})=0$, forms a Lie algebra in $ L( H, A)$. The exponential $ \exp^\star(\alpha) = \sum_{j\ge 0} \frac{1}{j!}\alpha^{\star j}$ is well defined and provides a bijection from $g_0$ onto the group $G_0 = \iota + g_0$ of linear maps, $\gamma$, which send the unit of $H$ to the algebra unit, $\gamma(\mathbf{1})=1_{A}$. An {\it{infinitesimal character}} with values in $A$ is a linear map $\xi \in L( H, A)$ such that for $x, y \in  H$, $\xi(xy) = 0$. The linear space of infinitesimal characters is denoted $g_{A} \subset g_0$. An $A$-valued map $\Phi$ in $ L( H, A)$ is called a {\it{character}} if $\Phi(\mathbf{1})=1_{ A}$ and for $x,y \in  H$, $\Phi(xy) = \Phi(x)\Phi(y)$. The set of characters is denoted by $G_{ A} \subset G_0$. It forms a pro-unipotent group for the convolution product with (pro-nilpotent) Lie algebra $ g_{ A}$. The exponential map $\exp^{\star}$ restricts to a bijection between $ g_{ A}$ and $G_{ A}$. The neutral element $\iota:=e_{ A}\circ \epsilon$  in $G_{ A}$ is given by $\iota(\mathbf{1})=1_{ A}$ and $\iota(x) = 0$ for $x \in \rm{Ker}(\varepsilon)=H^+$. The key property is the fact that the inverse of a character $\Phi \in G_{ A}$ is given by composition with the Hopf algebra antipode $S$, i.e., $\Phi^{\star -1} = \Phi \circ S$  \cite{GBFV,manchon2}.


\subsection{Output feedback Hopf algebra}
\label{sect:feedback}

Fliess operators and their interconnections are briefly reviewed next. For details the reader is referred to \cite{Ferfera_79,Ferfera_80,Fliess_81,Fliess_83,Gray-Duffaut_Espinosa_SCL11,Gray-Duffaut_Espinosa_FdB14,Gray-et-al_SCL14,Gray-Li_05,Gray-Wang_SCL02,Isidori_95,Thitsa-Gray_12}.

Iterated integrals play a central role in the analysis of nonlinear feedback control systems \cite{Isidori_95}.
It is well known that they come with a natural algebraic structure, i.e., products of iterated integrals can again be written as linear combinations of iterated integrals. Indeed, the classical integration by parts rule implies that the product of two Riemann integrals is given by
$$
	\int^t_0 u_1(x){\mathrm{d}}x  \int^t_0 u_2(y){\mathrm{d}}y
	= \int^t_0 \int^x_0 u_1(y) u_2(x){\mathrm{d}}y{\mathrm{d}}x
	+ \int^t_0 \int^x_0 u_1(x) u_2(y){\mathrm{d}}y{\mathrm{d}}x.
$$
This generalizes to the shuffle product for iterated integrals
$$
	f_n(u_1,\ldots,u_n)(t):=\int_{\Delta^n_{[0,t]}} u_1(a_1)  u_2(a_2) \cdots u_n(a_n) {\mathrm{d}}a_1 \cdots {\mathrm{d}}a_n,
$$
where $\Delta^n_{[0,t]}:=\{(a_1,\ldots, a_n),\  0\le a_1 \le \cdots \le a_n \le t\}$, as follows
\begin{align}
\label{ChenCl}
	& f_n(u_1,\ldots,u_n)(t)f_m(u_{n+1},\ldots,u_{n+m})(t) = \\
	& \quad \sum_{\sigma \in \smop{Sh}_{n,m}}
	\int\limits_{\Delta^{n+m}_{[0,t]}} u_{\sigma_1^{-1}}(a_1)  u_{\sigma_2^{-1}}(a_2)
	 \cdots u_{\sigma_{n+m}^{-1}}(a_{n+m}){\mathrm{d}}a_1 \cdots  {\mathrm{d}}a_{n+m}. \nonumber
\end{align}
Here $\mathrm{Sh}_{n,m}$ is the set of $(n,m)$-shuffles, i.e., permutations $\sigma $ on the set $[n+m]:=\{1,\ldots, n+m\}$, such that $\sigma_1<\cdots < \sigma_n$ and $\sigma_{n+1}<\cdots <\sigma_{n+m}$. It turns out to be rather convenient to abstract the product of iterated integrals in terms of a formal shuffle product on words. Reutenauer's monograph \cite{reutenauer} provides an exhaustive reference on this topic. Under concatenation, the set $X^*$ of words $\eta:=x_{i_1}\cdots x_{i_p}$ with letter $x_{i_j}$, $j \in \{1,\ldots ,p\}$ from the alphabet $X:=\{x_0,x_1\}$ forms a monoid. The empty word in $X^*$ containing no letters is denoted by~$\emptyset$, and the length of a word $\eta \in X^*$, written as $| \eta | \in \mathbb{N}_0$,
is equivalent to the number of letters it contains. The resulting graded vector space ${\mathbb Q}\langle X\rangle$, which is freely generated by words from $X^*$, is turned into a commutative graded algebra by defining the shuffle product of words
\begin{equation}
\label{shu-prod}
	x_{i_1}\cdots x_{i_p}\shuffle x_{i_{p+1}}\cdots x_{i_{p+q}}:=
		\sum_{\sigma\in\smop{Sh}_{p,q}} x_{\sigma_{i_1}^{-1}}\cdots x_{\sigma_{i_{p+q}}^{-1}}
\end{equation}
with $x_{i_j}\in X$, $j \in\{1,\ldots ,p+q\}$. The monomial $\un:=1\emptyset$ acts as a unit for the shuffle product so that
 $\un \shuffle \eta = \eta \shuffle \un =\eta$ for all $\eta \in X^*$. Product (\ref{shu-prod}) can be defined recursively as
\begin{align*}
	x_{i_1}\cdots x_{i_p}\shuffle x_{i_{p+1}}\cdots x_{i_{p+q}}
	&= x_{i_1}\big( x_{i_2}\cdots x_{i_p}\shuffle x_{i_{p+1}}\cdots x_{i_{p+q}} \big)\\
	&\quad\qquad + x_{i_{p+1}}\big( x_{i_1}\cdots x_{i_p}\shuffle x_{i_{p+2}}\cdots x_{i_{p+q}} \big), \nonumber
\end{align*}
or, equivalently, as
\begin{align*}
	x_{i_1}\cdots x_{i_p}\shuffle x_{i_{p+1}}\cdots x_{i_{p+q}}
	&= \big( x_{i_1}\cdots x_{i_{p-1}}\shuffle x_{i_{p+1}}\cdots x_{i_{p+q}} \big)x_{i_p}\\
	&\quad\qquad + \big( x_{i_1}\cdots x_{i_p}\shuffle x_{i_{p+2}}\cdots x_{i_{p+q-1}} \big)  x_{i_{p+q}}. \nonumber
\end{align*}
For example,
\begin{align*}
	x_{i_1} \shuffle x_{i_{2}}
		&= x_{i_1} x_{i_{2}} + x_{i_2}x_{i_{1}}\\
	x_{i_1} \shuffle x_{i_2} x_{i_{3}}
		&= x_{i_1} x_{i_{2}} x_{i_3} + x_{i_{2}} (x_{i_1} x_{i_{3}} + x_{i_3}x_{i_{1}}) .	
\end{align*}
The shuffle algebra $(\mathbb Q\langle X\rangle , \shuffle)$ is a connected graded Hopf algebra. The coproduct is defined by decatenation, that is, for any word $\eta=x_{i_1}\cdots x_{i_p}$
$$
	\Delta \eta := \eta \otimes \un + \un \otimes \eta + \sum_{l=1}^{p-1} x_{i_1}\cdots x_{i_{l}} \otimes x_{i_{l+1}}\cdots x_{i_p}.
$$
The antipode is given by the simple formula $S\eta= (-1)^{p} x_{i_p}\cdots x_{i_1}$.
In control theory, iterated integrals over a set of functions $u:=\{u_0,u_1\}$ are defined inductively by letting $E_{\emptyset}[u](t):=1$ and
$$
	E_{x_{i}\eta}[u](t):=\int_0^t u_{i}(s) E_{\eta}[u](s){\mathrm{d}}s,
$$
where $x_i\in X$ and $\eta\in X^\ast$.
The product \eqref{shu-prod} is homogeneous with respect to the length of words, and (\ref{ChenCl}) implies that for any $\eta, \xi \in X^\ast$
the corresponding product of iterated integrals satisfies
$$
	E_{\eta}[u](t)E_{\xi}[u](t) = E_{\eta \shuffle \xi}[u](t).
$$

Any mapping $c: X^\ast \rightarrow \mathbb{R}$ is called a formal power series with values in $\mathbb{R}$. The value of $c$ at $\eta \in X^\ast$ is written as $\langle c,\eta \rangle\in \mathbb{R}$ and called the coefficient of the word $\eta$ in $c$. The series $c$ is represented as the formal sum
$$
	c=\sum_{\eta\in X^\ast}\langle c,\eta \rangle\eta,
$$
and the collection of all formal power series over $X$ is denoted by $\mathbb{R}\langle\langle X \rangle\rangle$. It forms a unital associative $\mathbb{R}$-algebra under the catenation product and a unital associative and commutative $\mathbb{R}$-algebra under the shuffle product. A {\it{Fliess operator}} with generating series $c$ is defined to be the sum of iterated integrals over $u$ weighted by the coefficients of $c$, namely,
\begin{align*}
	F_c[u](t) := \sum_{\eta\in X^{\ast}} \langle c,\eta \rangle\,E_\eta[u](t).
\end{align*}
Given two Fliess operators $F_c$ and $F_d$ corresponding to $c,d \in \mathbb{R}\langle\langle X \rangle\rangle$, the parallel and product connections satisfy $F_c + F_d=F_{c+d}$ and $F_c F_d=F_{c \shuffle d}$, respectively \cite{Fliess_81}. When Fliess operators $F_c$ and $F_d$ are interconnected in a cascade fashion, the composite system $F_c \circ F_d$ has the Fliess operator representation $F_{c\circ d} \in \mathbb{R}\langle \langle X \rangle\rangle$, where the composition product
$$
	c\,  \circ\,  d := \sum_{\eta \in X^\ast}
	\langle c,\eta \rangle\, \psi_d(\eta)(\un)
$$
is defined through $\psi_{d}$, which is the continuous (in the ultrametric sense) algebra homomorphism from $\mathbb{R}\langle \langle X\rangle \rangle$ to ${\rm End}(\mathbb{R}\langle \langle X \rangle \rangle)$ uniquely specified by $\psi_{d}(x_i  \eta):=\psi_{d}(x_i) \circ \psi_{d}(\eta)$ with
\begin{equation*}
	\psi_{d}(x_0)({e}) := x_0e,\;\;\psi_{d}(x_1)({e}) := x_0(d \shuffle {e})
\end{equation*}
for any ${e} \in \mathbb{R}\langle \langle X \rangle \rangle$ \cite{Ferfera_79,Ferfera_80,Gray-Li_05}. By definition, $\psi_{d}(\un)$ is the identity map on $\mathbb{R}\langle \langle X\rangle \rangle$.

\begin{figure}[htb]
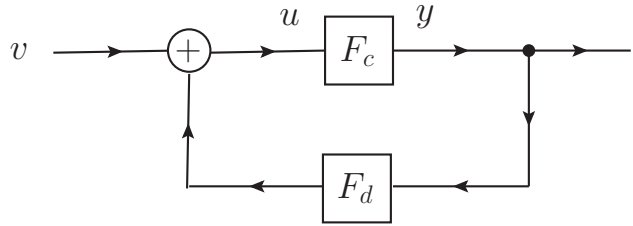

$$
	\feedbackcatd
$$
\caption{General output feedback system}
\label{fig:feedback-system-Fc-at-Fd}
\end{figure}

The output feedback connection of two Fliess operators as shown in Figure~\ref{fig:feedback-system-Fc-at-Fd} is much harder to characterize than the previous interconnections due to its recursive nature. While it is easily shown by fixed point arguments that the mapping $v\mapsto y$ has a Fliess operator representation, giving an explicit formula for its generating series, denoted here by the feedback product $c\, @\, d$, requires more sophisticated machinery \cite{Foissy_13,Gray-Duffaut_Espinosa_SCL11,Gray-Duffaut_Espinosa_FdB14,Gray-et-al_SCL14}. This is the origin of the output feedback Hopf algebra as described next.
First consider the set of operators
$$
	\mathcal{F}:=\{I+F_c: c \in \mathbb{R}\langle \langle X\rangle \rangle\},
$$
where $I$ denotes the identity operator. The symbol $\delta$ serves as the (fictitious) generating series for this identity map. That is, $F_\delta:=I$ such that $I+F_c := F_{\delta+c}=F_{c_\delta}$ with $c_\delta:=\delta+c$. The set of all such generating series for $\mathcal{F}$ will be denoted by $\mathbb{R}\langle \langle X_\delta \rangle \rangle$. The modified composition product of $c,d  \in \mathbb{R}\langle\langle X\rangle\rangle$ is defined by
$$
	c\, \tilde\circ\, d = \sum_{\eta \in X^\ast} \langle c,\eta \rangle\, \phi_d(\eta)(\un),
$$
where $\phi_{d}$ is the continuous (in the ultrametric sense) algebra homomorphism from $\mathbb{R}\langle \langle X\rangle \rangle$ to ${\rm End}(\mathbb{R}\langle \langle X \rangle \rangle)$ uniquely specified by $\phi_{d}(x_i  \eta):=\phi_{d}(x_i) \circ \phi_{d}(\eta)$ with
\begin{equation*}
	\phi_{d}(x_0)({e}) := x_0e,\quad \phi_{d}(x_1)(e) := x_1{e}+x_0(d \shuffle e),
\end{equation*}
for any ${e} \in \mathbb{R}\langle \langle X \rangle \rangle$, and $\phi_{d}(\un)$ is the identity map on $\mathbb{R}\langle \langle X \rangle \rangle$. It can be shown that
\begin{equation*}
	(x_0 c)\, \tilde{\circ}\, d = x_0(c\, \tilde{\circ}\, d),\quad
		(x_1 {c})\, \tilde\circ\, {d} = x_1 ({c}\, \tilde\circ\, {d}) +
		 x_0\big(d \shuffle ({c}\ \tilde\circ\ {d})\big). 	
\end{equation*}
The {\it{output feedback group}} is a fundamental object in this analysis. It is defined in terms of the group composition product on $\mathcal{F}$, namely,
$$
	F_{c_\delta}\circ F_{d_\delta}=(I+F_c)\, \circ\, (I+F_d) = F_{c_\delta\, \circ\, d_\delta},
$$
where $c_\delta \circ d_\delta := \delta+d+c \tilde\circ d$. Note that the same symbol will be used for composition on $\mathbb{R}\langle\langle X \rangle\rangle$ and composition on $\mathbb{R}\langle\langle X_\delta \rangle\rangle$. As elements in these two sets have a distinct notation, i.e., $c$ versus $c_\delta$, respectively, it will always be clear which product is at play. The generating series for the feedback group also forms a group as described in the next theorem.

\begin{thm} \label{th:allseriesdeltam-is-group}
The triple $(\mathbb{R}\langle\langle X_\delta \rangle\rangle,\circ,\delta)$ is a group with unit $\delta$.
\end{thm}

For any word $\eta \in X^\ast$, the coordinate functions, $a_\eta$, on the feedback group in Theorem \ref{th:allseriesdeltam-is-group} are defined as elements of the dual space $\mathbb{R}^*\langle \langle X_\delta \rangle \rangle$, i.e., maps on formal series $c_\delta= \delta + c =  \delta + \sum_{\eta \in X}\langle{c},\eta\rangle \eta \in \mathbb{R}\langle \langle X_\delta \rangle \rangle$ giving coefficients of words $\eta \in X^*$
$$
	a_\eta(c) := \langle c , \eta \rangle.
$$
In this context, $a_\delta$ denotes the coordinate function with respect to $\delta$, i.e., $a_\delta(\delta)=1$, and zero otherwise.
The feedback group can be considered to be the group of characters with respect to a connected graded Hopf algebra underlying the coordinate functions. This
is emphasized by the notation $c_\delta (a_\eta):=a_\eta(c_\delta )$ for $\eta \in X^* \cup \{\delta \}$. However, the reader is warned that henceforth the coefficient function $a_\delta$ is identified with the symbol $\un$ such that $\un(\delta):=1$. This Hopf algebraic point of view, presented next, originated in a series of papers starting with \cite{Gray-Duffaut_Espinosa_SCL11,Gray-et-al_SCL14} and continued in \cite{DuffautEFG_2014,Foissy_13,Gray-Duffaut_Espinosa_FdB14}, where the Fa\`a di Bruno Hopf algebra defined in terms of these coordinate functions was developed. The link to the classical Fa\`a di Bruno Hopf algebra \cite{manchonfrabetti,GBFV} follows from the compositional nature of the feedback group in Theorem \ref{th:allseriesdeltam-is-group}.

Let $V$ denote the $\mathbb{R}$-vector space spanned by the coordinate functions. For a word $\eta\in X^\ast$, let $|\eta|_{x_0}$ and $|\eta|_{x_1}$ denote, respectively,
the number of times the letters $x_0$ and $x_1$ appear in the word.
If the {\em degree} of $a_{\eta}$ is defined as $\deg(a_{\eta})=\deg(\eta)=2|\eta|_{x_0}+|\eta|_{x_1}+1$ and $\deg(\un)=0$, then $V$ is a connected graded vector space, that is, $V=\bigoplus_{n\geq 0} V_n$
with
$$
	V_n=\big\langle a_\eta:\deg(a_\eta)=n\big\rangle_\mathbb{R}
$$
and $V_0=\mathbb{R}\un$. It is rather striking to see that Devlin's choice of grading coincides with that of Foissy in \cite{Foissy_13}. This indicates how natural the point of view of graded connected Hopf algebras matches the algebraic underpinnings of Devlin's work. Consider next the free unital commutative $\mathbb{R}$-algebra, $H$, with product
\begin{equation*}
	\mu:	a_\eta\otimes a_{\xi} \mapsto a_{\eta}a_{\xi},
\end{equation*}
and unit $\un$, which is induced by the counit identity $\epsilon(a_{\eta}):=\un(a_{\eta})=0$ for all $\eta\in X^\ast$.
This product is associative. The graduation on $V$ induces a connected graduation on $H$ with $\deg(a_\eta a_\xi)=\deg(a_\eta)+\deg(a_\xi)$.
Specifically, $H=\bigoplus_{n\geq 0} H_n$, where $H_n$ is the set of all elements of degree $n$ and $H_0=\mathbb{R}\un$. Moreover, by definition
\begin{equation}
\label{character}
	a_{\eta_1} \cdots a_{\eta_l} (c) := a_{\eta_1}(c) \cdots a_{\eta_l}(c)
\end{equation}
for $c_\delta = \delta + c \in \mathbb{R}\langle \langle X_\delta \rangle \rangle$.

To describe the coalgebra on $H$, first consider the left- and right-augmentation operators
$\theta_{i}$ and $\tilde{\theta}_{i}$, respectively, where $i=0,1$. These are endomorphisms on $V$ which concatenate letters
in the following manner:
$$
	\theta_{i}(a_\eta):=a_{x_i\eta}, \quad\quad \tilde\theta_{i}(a_\eta):=a_{\eta x_i},
$$
and thereby increase the degree of an element in $V$ by either one ($i=1$) or two ($i=0$). For any word $\eta= x_{i_1} \cdots x_{i_l}$, observe that
$$
	\theta_\eta (a_\emptyset):= \theta_{{i_1}} \circ \cdots \circ \theta_{{i_l}} (a_\emptyset)
	=\tilde\theta_\eta (a_\emptyset)
	:= \tilde\theta_{{i_l}} \circ \cdots \circ \tilde\theta_{{i_1}} (a_\emptyset).
$$	
In particular, note that $\theta_j \tilde\theta_i = \tilde\theta_i \theta_j$ for $i,j=0,1$. Both the right- and left-augmentation operators are by definition extended to derivations on the algebra $H$ so that
$$
	 \tilde\theta_{i}(a_{\eta_1} \cdots a_{\eta_l})
	 := \sum_{k=1}^l a_{\eta_1} \cdots a_{\eta_k x_i} \cdots a_{\eta_l}
$$
and analogously for left-augmentation. On the unit $\un$ of $H$, $\tilde\theta_{i}(\un)=\theta_{i}(\un):=0$.
The deshuffling coproduct $\Delta_\shuffle V^+ \subset V^+\otimes V^+$ is defined by
\begin{subequations}
\label{eq:shuffle-coproduct-induction}
\allowdisplaybreaks{
\begin{align}
	\Delta_{\shuffle} a_{\emptyset}&=a_{\emptyset}\otimes a_{\emptyset} \\
	\Delta_{\shuffle}\circ\theta_k&=(\theta_k\otimes id + id \otimes \theta_k)\circ\Delta_{\shuffle}\\
	\Delta_{\shuffle}\circ\tilde\theta_k&=(\tilde\theta_k\otimes id + id \otimes \tilde\theta_k)\circ\Delta_{\shuffle}. \label{deshuffle}
\end{align}}%
\end{subequations}
For example, the first few terms of $\Delta_{\shuffle}$ are:
\allowdisplaybreaks{
\begin{align*}
	\Delta_{\shuffle} a_{\emptyset}&=a_{\emptyset}\otimes a_{\emptyset} \\
	\Delta_{\shuffle} a_{x_{i_1}}&=a_{x_{i_1}}\otimes a_{\emptyset}+a_{\emptyset}\otimes a_{x_{i_1}} \\
	\Delta_{\shuffle} a_{x_{i_2}x_{i_1}}&=a_{x_{i_2}x_{i_1}}\otimes a_{\emptyset}
				+a_{x_{i_2}}\otimes a_{x_{i_1}}
				+a_{x_{i_1}}\otimes a_{x_{i_2}}+ a_{\emptyset}\otimes a_{x_{i_2}x_{i_1}}.
\end{align*}}%
In the following, Sweedler's notation \cite{Sweedler_69} is employed for the deshuffling coproduct so that
$\Delta_\shuffle a_\eta=a_{\eta'}\otimes a_{\eta''}$.

The first of two coproducts is defined on $H$, namely, the coproduct $\tilde{\Delta}H\subset V\otimes H$ given by the identity
\begin{align*}
	\tilde{\Delta} a_{\eta}(c,d) 	&=a_{\eta}(c \ \tilde\circ\ d)=\langle c \ \tilde\circ\ d,\eta \rangle \nonumber \\
						&=:\sum a_{\eta(1)}(c)\;a_{\eta(2)}(d) \nonumber \\
						&=\sum a_{\eta(1)}\otimes a_{\eta(2)}(c,d) \label{eq:tilde-delta-identity},
\end{align*}
where $c,d \in \mathbb{R}\langle \langle X\rangle \rangle$, $a_{\eta(1)}\in V$ and $a_{\eta(2)} \in H$. Note that $a_{\eta(2)}(d) $ is defined according to \eqref{character}. The summation is taken over all terms for which $\langle c \ \tilde\circ\ d,\eta \rangle$ is different from zero. The reader should note that a second Sweedler notation is employed for this coproduct, specifically, $\tilde{\Delta}a_{\eta}=\sum a_{\eta(1)}\otimes  a_{\eta(2)}$, which should not be confused with the one used for the deshuffling coproduct given above. The key observation here is that the coproduct $\tilde{\Delta}$ can be computed inductively as described in the following lemma.

\begin{lem} \label{le:tilde-delta-inductions}
The following identities hold:
\begin{enumerate}

\item
	$\tilde{\Delta} a_\emptyset=a_{\emptyset} \otimes \un$

\item
	$\tilde{\Delta}\circ \theta_1= (\theta_1 \otimes \id) \circ \tilde{\Delta}$

\item
	$\tilde{\Delta}\circ \theta_0=(\theta_0 \otimes \id) \circ \tilde{\Delta}
	+ (\theta_1 \otimes \mu)\circ(\tilde{\Delta}\otimes \id)\circ \Delta_{\shuffle}$,

\end{enumerate}
where $\id$ denotes the identity map on $H$.
\end{lem}

The {\em full coproduct} on $H$ is defined as
\begin{equation}
\label{def:fullcoprod1}
	\Delta a_\eta:= \tilde{\Delta}a_\eta + \un \otimes a_\eta,
\end{equation}
while the {\em reduced} coproduct is given by $\Delta'a_\eta :=\tilde{\Delta}a_\eta -  a_\eta \otimes \un$.
The following theorem summarizes the properties of the output feedback Hopf algebra $H$.

\begin{thm} \label{th:output-feedback-HA}
$(H,\mu,\Delta,\varepsilon,S)$ is a connected graded commutative non-cocommutative Hopf algebra.
\end{thm}

The central object used for defining the feedback product is the antipode $S$ on $H = \bigoplus_{n\ge 0} H_n$.  It satisfies the identity $a_\eta(\gsc^{\circ -1})=Sa_\eta(\gsc)$ for all $\eta \in X^\ast$, where the feedback group inverse is given by $c_\delta^{\circ -1}=\delta+c^{\circ -1}$, which is calculated by composing $c_\delta$ as a Hopf algebra character with the Hopf algebra antipode $S$. This statement together with \eqref{character} implies that, in Hopf algebraic terms, the feedback group in Theorem \ref{th:allseriesdeltam-is-group} is the group of Hopf algebra characters over $H$.
From Theorem \ref{th:antipode-induction}, for any $a_\eta \in H^+$, that is, $\deg(a_\eta) > 0$, the antipode can be computed by
\begin{equation}
	S a_\eta=-a_\eta -{\sum}' (S a_{\eta(1)})a_{\eta(2)}=-a_\eta - {\sum}'  a_{\eta(1)}S a_{\eta(2)}.
\label{eq:general-antipode-induction}
\end{equation}
Hence, a totally inductive algorithm to compute the antipode for the output feedback group is given as follows.

\begin{thm} \cite{Gray-et-al_MTNS14}
The antipode $S$ applied to any $a_\eta \in V_k$, $k \geq 1$ in the output feedback Hopf algebra $H$ can be computed by the following algorithm:
\begin{description}
\item[\hspace*{0.25in}i] 	Inductively compute $\Delta_{\shuffle}$ via \eqref{eq:shuffle-coproduct-induction}.
\item[\hspace*{0.2in}ii] 	Inductively compute $\tilde{\Delta}$ via Lemma~\ref{le:tilde-delta-inductions}.
\item[\hspace*{0.17in}iii] 	Inductively compute $S$ via \eqref{eq:general-antipode-induction}.
\end{description}
\end{thm}

Using this algorithm, the first few antipode terms are found to be:
\begin{subequations}
\label{eq:antipode-terms}
\begin{align}
H_0&:S\un = \un \\
H_1&:S a_\emptyset=-a_\emptyset \\
H_2&:S a_{x_1}=-a_{x_1} \\
H_3&:S a_{x_0}=-a_{x_0}+ a_{x_1}a_\emptyset \\
H_3&:S a_{x_1x_1}=-a_{x_1x_1} \\
H_4&:S a_{x_0x_1}=-a_{x_0x_1}+a_{x_1}a_{x_1}+a_{x_1x_1}a_\emptyset \\
H_4&:S a_{x_1x_0}=-a_{x_1x_0}+ a_{x_1x_1}a_{\emptyset} \\
H_4&:S a_{x_1x_1x_1}=-a_{x_1x_1x_1} \\
H_5&:S a_{x_0x_0}=-a_{x_0x_0}+a_{x_1}a_{x_0}+a_{x_1 x_0}a_\emptyset+ a_{x_0x_1}a_\emptyset-
a_{x_1}a_{x_1} a_{\emptyset}-a_{x_1x_1}a_{\emptyset}a_{\emptyset}.
\end{align}
\end{subequations}

Finally, an explicit formula for the feedback product is given in the next theorem in terms of the
output feedback antipode. It is this theorem that is exploited in the next section to provide
a combinatorial interpretation of Devlin's polynomials.

\begin{thm}
\label{th:feedback-product-formula}
For any $c,d\in\allseries$ it follows that
$$
	c\, @ \, d=c\, \tilde{\circ}\, ((-d)\, \circ\,  c)^{\circ -1}.
$$
\end{thm}

It is worth noting that this formula holds even when $d$ is formally replaced with $\delta$, or equivalently, $F_d$ is replaced with $F_\delta=I$ in Figure~\ref{fig:feedback-system-Fc-at-Fd}. This so called {\em unity feedback system}, a central notion in control theory, has a closed-loop Fliess operator representation with generating series
\begin{equation}
\label{eq:generating-series-unity-feedback-system}
	c\, @\, \delta=(-c)^{\circ -1},
\end{equation}
since in general $e:=c\, @\, \delta$ must satisfy the fixed point equation $e=c\, \tilde{\circ}\,  e$, while at the same time from Theorem~\ref{th:feedback-product-formula} it follows that $e=c\, \tilde{\circ}\,  (-c)^{\circ -1}$.
Finally, results such as Lemma \ref{le:tilde-delta-inductions}, Theorem~\ref{th:output-feedback-HA} and Theorem \ref{th:feedback-product-formula}
can be easily extended to the multivariable case, that is, the situation where the alphabet $X:=\{x_0,x_1,\ldots,x_m\}$ has $m\geq 2$ letters \cite{Gray-et-al_SCL14}.


\section{Devlin's polynomials from the feedback antipode}
\label{sect:Devlin-polynomials}

The starting point for a combinatorial interpretation of Devlin's work is the re-writing of (\ref{eq:Devlineq-without-gamma}) in terms of the shuffle product. Consider a Fliess operator with input functions $u_0=\alpha$ and $u_1=\beta$. In which case, if $z=F_e[u]$ for some generating series $e\in\re\langle\langle X\rangle\rangle$ then (\ref{eq:Devlineq-without-gamma}) is equivalent to
\begin{displaymath}
	u_0 F_{x_0^{-1}(e)}+u_1 F_{x_1^{-1}(e)}=u_0 F_{e^{\shuffle\,3}}+u_1 F_{e^{\shuffle\,2}},
\end{displaymath}
where the left-shift operator, $x_i^{-1}(\cdot)$, is defined on $X^\ast$ by $x_i^{-1}(x_i\eta)=\eta$ with $\eta\in X^\ast$ and zero otherwise. It is assumed to act linearly on series.
Note that the left-shift operator may be regarded as the adjoint of the left-augmentation map, that is,  $\langle a_{x_i\eta}, x_j \nu \rangle =\langle \theta_i a_\eta, x_j \nu \rangle = \langle a_\eta, x_i^{-1} (x_j \nu) \rangle =\delta_{i,j} \langle a_\eta,  \nu \rangle.$ Here $\delta_{i,j}$ denotes the usual Kronecker delta.
Analogously, the right-shift operator $\tilde{x}_i^{-1}(\cdot)$ is defined on $X^\ast$ by $\tilde{x}_i^{-1}(\eta x_j)=\delta_{i,j}\eta$ for $\eta \in X^\ast$.
Since generating series are unique and the functions $u_i$ are not fixed, this expression is equivalent to the pair of shuffle equations
\begin{displaymath}
\label{eq:Devlin-shuffle-equations}
	x_0^{-1}(e)=e^{\shuffle\,3},\;\;\;
	x_1^{-1}(e)=e^{\shuffle\,2},
\end{displaymath}
which can be combined to give
\begin{equation}
\label{eq:feedback-eqn-e}
	e=(e,\emptyset)+x_0e^{\shuffle\,3}+x_1e^{\shuffle\,2}.
\end{equation}
From here the strategy is to solve this shuffle equation for $e$ in order to determine an explicit expression for the first return map (\ref{eq:first-return-map}). The first observation is that if $c$ is the Ferfera series (\ref{FFsystem}), then its corresponding input-output map $y=F_c[u]$ has a simple state space realization. In light of the shuffle algebra identity $k!\,x_1^k=x_1^{\shuffle k}$, it follows directly that
\begin{displaymath}
	y(t) = \sum^{\infty}_{k = 0} E^k_{x_1}[u](t)= \big(1 - E_{x_1}[u](t)\big)^{-1}.	
\end{displaymath}
Defining the state $z=y$, then
 $\dot{z}(t) = \big(1 - E^k_{x_1}[u](t)\big)^{-2}u(t)=z^2(t)u(t)$. So $y=F_c[u]$ has the one dimensional state space realization
\begin{equation}
\label{FFsystemState}
	\begin{array}{ll}
	\dot{z}(t) = z^2(t)u(t), 	& z(0)=1\\
	y(t)=z(t).			&
	\end{array}
\end{equation}%
\begin{figure}[htb]
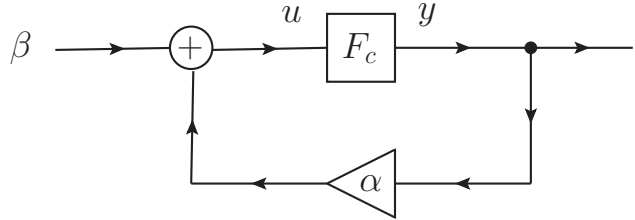

$$
	\feedbackA
$$
\caption{Output feedback system corresponding to (\ref{DevlinsystemState})}
\label{fig:Devlin-feedback-system}
\end{figure}%
Next system \eqref{FFsystemState} is put in an output feedback loop as shown in Figure~\ref{fig:Devlin-feedback-system} so that $u = \beta +\alpha y$. The function $y$ must therefore satisfy the feedback equation $y=F_c[\beta + \alpha y]$. Equivalently, the closed-loop system has the state space
realization
\begin{equation}
\label{DevlinsystemState}
	\begin{array}{ll}
		\dot{z}(t) = \alpha(t) z^3(t) + \beta(t) z^2(t),		& z(0)=1\\
		y(t)=z(t). 								&
	\end{array}		
\end{equation}
The solution to the feedback equation in integral form is
\begin{displaymath}
	z(t)=1+\int_0^t \alpha(\tau) z^3(\tau)\,{\rm d}\tau+\int_0^t \beta(\tau) z^2(\tau)\,{\rm d}\tau,
\end{displaymath}
which is equivalent to $z(t)=F_e[u]$ if $e$ satisfies the shuffle equation (\ref{eq:feedback-eqn-e}). In control theory it is most common to have $u_0=\alpha=1$. But from a combinatorial perspective, nothing is changed if $\alpha$ is left arbitrary. In which case, this feedback system can be identified with a unity feedback system from which the next main result follows. It is useful here to define the projection $(\cdot)_n:\re\langle\langle X\rangle\rangle\rightarrow \re\langle X\rangle: c\mapsto c_n$, where $c_n=\sum_{\eta\in X^*_n} \langle c,\eta \rangle \eta$, and $X^*_n$ is the set of all words in $X^*$ with degree $n$.

\begin{thm} \label{thm:Devlin-polynomials}
The polynomials
\begin{displaymath}
	a_n:=((-c)^{\circ -1})_n=\sum_{\eta\in X^*_n} Sa_\eta(-c)\eta,\quad n\geq 1
\end{displaymath}
with $c=\sum_{k\geq 0} k!\,x_1^k$ satisfy
\begin{displaymath}
	a_n(t)=F_{a_n}[u](t),
\end{displaymath}
where $u:=\{\alpha, \beta\}$, and the functions $a_n(t)$ are the components of the formal solution \eqref{eq:fomal-solution-Abel-equation} of the Abel equation \eqref{eq:Devlineq-without-gamma}.
\end{thm}
\begin{proof}
Recall that the elements in the feedback group are the characters with respect to $H$ such that $Sa_\eta(-c) = (-c_\delta)\circ Sa_\eta = \langle (-c)^{\circ -1},\eta \rangle$ with $-c = \sum_{k\geq 0} (-k!)\,x_1^k$. Setting $z(0)=r=1$ in (\ref{eq:fomal-solution-Abel-equation}) it is clear that the solution $z(t)=\sum_{n=1}^\infty a_n(t)$ of the Abel equation \eqref{eq:Devlineq-without-gamma} can be identified with the output of state space realization (\ref{DevlinsystemState}). Since this system has a Fliess operator representation with generating series $(-c)^{\circ -1}$ as per (\ref{eq:generating-series-unity-feedback-system}), the claim follows immediately.
\end{proof}

\begin{exam}
Suppose $n=4$. Then
$$
	a_4=\langle (-c)^{\circ -1},x_0x_1\rangle x_0x_1
		+\langle (-c)^{\circ -1},x_1x_0\rangle x_1x_0
		+\langle (-c)^{\circ -1},x_1x_1x_1\rangle x_1x_1x_1,
$$
and from \eqref{eq:antipode-terms} it follows that
\begin{align*}
\langle (-c)^{\circ -1},x_0x_1\rangle&=Sa_{x_0x_1}(-c)=a_{x_1}(-x_1)a_{x_1}(-x_1)+a_{x_1x_1}(-2!x_1x_1)a_{\emptyset}(-1)=3 \\
\langle (-c)^{\circ -1},x_1x_0\rangle&=Sa_{x_1x_0}(-c)=a_{x_1x_1}(-2!x_1x_1)a_{\emptyset}(-1)=2 \\
\langle (-c)^{\circ -1},x_1x_1x_1\rangle&=Sa_{x_1x_1x_1}(-c)=-a_{x_1x_1x_1}(-3!\,x_1x_1x_1)=6,
\end{align*}
which agrees with \eqref{eq:Devlin-polynomial-a4}.
\end{exam}

If the only goal is to simply compute the coefficients of the polynomials $a_n$, control theorists will recognize a more direct route using the realization
(\ref{DevlinsystemState}). Consider any input-output map $y=F_c[u]$ with a finite dimensional control-affine state space
realization of the form
\begin{displaymath}
	\begin{array}{ll}
		\dot{z}= g_0(z)+g_1(z)u_1,	&z(0)=z_0\\
		y= h(z),					&
	\end{array}
\end{displaymath}
where each $g_j$ and $h$ is an analytic vector field and function, respectively,
on some neighborhood $W$ of $z_0$.
Then the generating series $c$ can be differentially generated from the vector fields via
\begin{equation*}
	\langle c,\eta \rangle=L_{g_{\eta}}h(z_0),\quad \forall\eta\in X^\ast, \label{eq:c-equals-Lgh}
\end{equation*}
where
\begin{displaymath}
	L_{g_{\eta}}h:=L_{g_{j_1}}\cdots L_{g_{j_k}}h, \quad \eta=x_{j_k}\cdots x_{j_1},
\end{displaymath}
the {\em Lie derivative} of $h$ with respect to $g_j$ is defined as
\begin{displaymath}
	L_{g_j} h: W\rightarrow \re: z \mapsto \frac{\partial h}{\partial z}(z)\, g_j(z),
\end{displaymath}
and $L_{g_\emptyset}h=h$ \cite{Fliess_83,Isidori_95,Wang_90}. Therefore, setting $g_0(z)=z^3$ and $g_1(z)=z^2$, formula (\ref{eq:Devlin-coefficients}) follows immediately. While this approach does not provide any direct insight into the underlying combinatorial structures that are present, it does inspire the new general antipode recursion described in the next section. This is somewhat surprising given that not every $c\in\re \langle\langle X\rangle\rangle$ has {\em local coordinates} in the sense that it is differentially generated by a set of vector fields \cite{Fliess_83}.


\section{New recursion for the feedback antipode}
\label{sect:new-antipode-recursion}

Given the standard antipode recursions in Theorem~\ref{th:antipode-induction} for connected graded Hopf algebras and the identity in Theorem~\ref{thm:Devlin-polynomials}
relating Devlin's polynomials to such an antipode, it seems plausible that Devlin's recursion (\ref{eq:Devlin-recursion}) is simply one such antipode recursion in some specific
form. Interestingly, this appears not to be the case. In light of Lemma~\ref{le:tilde-delta-inductions}, it is clear that the left-augmentation maps play a key role in the
standard antipode recursions, but in this section a different antipode recursion will be developed, one in which the right-augmentation map is the main player. Not only can
this new recursion be linked directly to (\ref{eq:Devlin-recursion}), as described in the next section, but it also provides some new insights into the output feedback
antipode itself as well as a simple algorithm for computer implementation.

The first theorem provides an alternative to the recursion in Lemma~\ref{le:tilde-delta-inductions} for computing the coproduct
of the output feedback Hopf algebra.
It is useful to introduce
the multiplication map $\kappa_\eta$ on $H$ for any $\eta\in X^\ast$ with the defining property that
$$
	\kappa_\eta(a_\nu) := a_\nu a_\eta, \quad \nu\in X^\ast.
$$	
Note, in particular, that $\kappa_\eta(\un) = a_\eta$ and $ \kappa_\emptyset \circ \tilde\theta_i  (a_\eta) = a_{\eta x_i} a_\emptyset$, whereas from the Leibniz rule
$$
	\tilde\theta_i \circ \kappa_\emptyset (a_\eta) = a_i a_\eta + a_{\eta x_i} a_\emptyset.
$$

\begin{thm}
For any nonempty word $\eta= x_{i_1} \cdots x_{i_l}$, the coproduct \eqref{def:fullcoprod1} is given by
\begin{equation}
\label{def:fullcoprod2}
	\Delta a_\eta = \Delta \tilde\theta_\eta (a_\emptyset) = \tilde\Theta_{\eta} \circ \Delta a_\emptyset,
\end{equation}
where $\tilde\Theta_{\eta} := \tilde\Theta_{i_l} \circ \cdots \circ  \tilde\Theta_{i_1}$,
$$
	\tilde\Theta_{0} := (\tilde\theta_{0} \otimes \id + \id \otimes \tilde\theta_{0}  + \tilde\theta_{1} \otimes \kappa_\emptyset ),
$$
and
$$
	\tilde\Theta_{1} := (\tilde\theta_{1} \otimes \id + \id \otimes \tilde\theta_{1}).
$$
\end{thm}

\begin{proof}
The proof is via the degree of $a_\eta$.
Indeed, with $\Delta a_\emptyset = a_{\emptyset} \otimes \un + \un \otimes a_{\emptyset}$ and recalling that
$\tilde\theta_i (\un)=0$, it is immediate that
\begin{displaymath}
	\tilde\Theta_{1} \circ\Delta a_\emptyset 	= a_{x_1} \otimes \un + \un \otimes a_{x_1}
								= \Delta a_{x_1}.
\end{displaymath}
So the claim holds when $\deg(a_\eta)=1$. Before proceeding to the general case, it is instructive to compute
a few more lower degree terms. For example,
\allowdisplaybreaks{
\begin{align*}	
	\tilde\Theta_{0} \circ\Delta a_\emptyset &= a_{x_0} \otimes \un
			+ \un \otimes a_{x_0} + a_{x_1} \otimes a_{\emptyset} =\Delta a_{x_0}\\
	\tilde\Theta_{1}\circ\tilde\Theta_{1}  \circ\Delta a_\emptyset
		&= (\tilde\theta_{1} \otimes \id + \id \otimes \tilde\theta_{1}) ^2 \circ\Delta a_\emptyset \\
		&= (\tilde\theta_{1}\tilde\theta_{1} \otimes \id + 2 \tilde\theta_{1} \otimes \tilde\theta_{1}
				+ \id \otimes \tilde\theta_{1}\tilde\theta_{1})\circ \Delta a_\emptyset \\
		&= a_{x_1x_1} \otimes \un + \un \otimes a_{x_1x_1}=\Delta a_{x_1x_1}\\
		\tilde\Theta_{0}\circ\tilde\Theta_{1} \circ \Delta a_\emptyset
		&= (\tilde\theta_{0} \otimes \id + \id \otimes \tilde\theta_{0}
		+ \tilde\theta_{1} \otimes \kappa_\emptyset )\circ
			(\tilde\theta_{1} \otimes \id + \id \otimes \tilde\theta_{1})\circ\Delta a_\emptyset \\
		&= (\tilde\theta_{0}\tilde\theta_{1} \otimes \id + \id \otimes \tilde\theta_{0}\tilde\theta_{1}
				+ \tilde\theta_{0} \otimes \tilde\theta_{1} + \tilde\theta_{1} \otimes \tilde\theta_{0}\\
		& \quad		+  \tilde\theta_{1}\tilde\theta_{1}  \otimes \kappa_\emptyset
		+ \tilde\theta_{1} \otimes \kappa_\emptyset \tilde\theta_{1})(a_{\emptyset} \otimes \un
		+ \un \otimes a_{\emptyset})\\
		&= \tilde\theta_{0}\tilde\theta_{1} a_{\emptyset} \otimes \un
			+ \un \otimes \theta_{0}\tilde\theta_{1} a_\emptyset
			+ \tilde\theta_{1}\tilde\theta_{1} a_\emptyset  \otimes \kappa_\emptyset \un \\
		&= a_{x_1x_0} \otimes \un + \un \otimes a_{x_1x_0} + a_{x_1x_1} \otimes a_{\emptyset}
			= \Delta a_{x_1x_0}  \\
	\tilde\Theta_{1}\circ\tilde\Theta_{0} \circ\Delta a_\emptyset	
		&= (\tilde\theta_{1} \otimes \id + \id \otimes \tilde\theta_{1})
			\circ(\tilde\theta_{0} \otimes \id + \id \otimes \tilde\theta_{0}
				+ \tilde\theta_{1} \otimes \kappa_\emptyset )\circ\Delta a_\emptyset \\
		&= (\tilde\theta_{1}\tilde\theta_{0} \otimes \id + \id \otimes \tilde\theta_{1}\tilde\theta_{0}
				+ \tilde\theta_{0} \otimes \tilde\theta_{1} + \tilde\theta_{1} \otimes \tilde\theta_{0}\\
		& \quad		+  \tilde\theta_{1}\tilde\theta_{1}  \otimes \kappa_\emptyset
		+ \tilde\theta_{1} \otimes \tilde\theta_{1} \kappa_\emptyset )(a_{\emptyset} \otimes \un
		+ \un \otimes a_{\emptyset})\\
		&= \tilde\theta_{0}\tilde\theta_{1} a_{\emptyset} \otimes \un
			+ \un \otimes \tilde\theta_{0}\tilde\theta_{1} a_\emptyset
			+ \tilde\theta_{1}\tilde\theta_{1} a_\emptyset  \otimes \kappa_\emptyset \un
			+ \tilde\theta_{1}a_\emptyset  \otimes \tilde\theta_{1} \kappa_\emptyset \un \\
		&= a_{x_0x_1} \otimes \un + \un \otimes a_{x_0x_1}
			+ a_{x_1x_1} \otimes a_{\emptyset}
			+ a_{x_1} \otimes a_{x_1}
			= \Delta a_{x_0x_1}.%
\end{align*}}%
Note that in several places above the parentheses in expressions such as $\tilde\theta_{0}\tilde\theta_{1} a_{\emptyset}$ have been omitted for notational simplicity. The same simplification is applied when convenient further below.
Now assume that (\ref{def:fullcoprod2}) holds up to some fixed degree $n\geq 1$.
The following compact notation
$$
	\tilde\Theta_{i} = (\tilde\theta_{i} \otimes \id + \id \otimes \tilde\theta_{i}
	+ \delta_{0,i} \tilde\theta_1 \otimes \kappa_\emptyset)
$$
is used.
There are two cases to consider. First suppose that $\deg(a_{x_1 \eta}) =n$ so that $\deg(a_{x_1 \eta x_i})>n$, $i=0,1$. Using \eqref{def:fullcoprod1} and
Lemma~\ref{le:tilde-delta-inductions}, it is clear that
$$
	\Delta a_{x_1 \eta x_i} = (\theta_1 \otimes \id)\circ\tilde\Delta a_{\eta x_i} + \un \otimes a_{x_1 \eta x_i}.
$$
Employing, in addition, the induction hypothesis, the identity $\theta_i(\un)=0$, and the fact that the left- and right-augmentation operators commute on $V$, i.e., $\theta_j \tilde\theta_i = \tilde\theta_i \theta_j$ for $i,j=0,1$,
it then follows that
\allowdisplaybreaks{
\begin{align*}
	\Delta a_{x_1 \eta x_i}&= (\theta_1 \otimes \id)\circ(\Delta a_{\eta x_i} - \un \otimes a_{\eta x_i} ) + \un \otimes a_{x_1 \eta x_i} \\
	&= (\theta_1 \otimes \id)\circ\Delta a_{\eta x_i} + \un \otimes a_{x_1 \eta x_i} \\
	&= (\theta_1 \otimes \id) \circ(\tilde\theta_{i} \otimes \id + \id \otimes \tilde\theta_{i} + \delta_{0,i} \tilde\theta_1 \otimes \kappa_\emptyset)
			\circ\Delta a_\eta + (\id \otimes \tilde\theta_i)(\un \otimes a_{x_1\eta})\\
	&= (\tilde\theta_{i} \otimes \id + \id \otimes \tilde\theta_{i}+ \delta_{0,i} \tilde\theta_1 \otimes \kappa_\emptyset)
		\circ(\theta_1 \otimes \id)\circ\Delta a_\eta + (\id \otimes \tilde\theta_i)(\un \otimes a_{x_1\eta})\\
	&= (\tilde\theta_{i} \otimes \id + \id \otimes \tilde\theta_{i}+ \delta_{0,i} \tilde\theta_1 \otimes \kappa_\emptyset)
		\Big((\theta_1 \otimes \id)\circ\tilde\Delta a_\eta + \un \otimes a_{x_1\eta}\Big)\\
	&= (\tilde\theta_{i} \otimes \id
		+ \id \otimes \tilde\theta_{i}+ \delta_{0,i} \tilde\theta_1 \otimes \kappa_\emptyset)\circ\Delta a_{ x_1\eta}\\
	&=\tilde\Theta_{x_1\eta x_i}\circ \Delta a_\emptyset.%
\end{align*}}%
Now suppose instead that $\deg(a_{x_0 \eta})=n$ so that $\deg(a_{x_0 \eta x_i})>n$, $i=0,1$.
Sweedler's notation for the deshuffling coproduct \eqref{deshuffle}, i.e., $\Delta_\shuffle a_\eta=a_{\eta'}\otimes a_{\eta''}$, is useful here.
Analogous to the previous case, observe
\allowdisplaybreaks{
\begin{align*}
	\Delta a_{x_0 \eta x_i}& = (\theta_0 \otimes \id)\circ\tilde\Delta a_{\eta x_i} + \un \otimes a_{x_0 \eta x_i}
				+ (\theta_1 \otimes \mu)\circ(\tilde{\Delta}\otimes \id)\circ \Delta_{\shuffle}a_{\eta x_i}\\
		&=  (\tilde\theta_{i} \otimes \id + \id \otimes \tilde\theta_{i}
				+ \delta_{0,i} \tilde\theta_1 \otimes \kappa_\emptyset)
				\big((\theta_0 \otimes \id)\circ\tilde\Delta a_\eta + \un \otimes a_{x_0\eta}\big)\\
		& \quad + (\theta_1 \otimes \mu)\circ(\tilde{\Delta}\otimes \id)\circ (\tilde\theta_{i} \otimes \id
				+ \id \otimes \tilde\theta_{i} )
		 		\circ\Delta_{\shuffle}a_{\eta}\\
		&=  (\tilde\theta_{i} \otimes \id + \id \otimes \tilde\theta_{i}
			+ \delta_{0,i} \tilde\theta_1 \otimes \kappa_\emptyset)
			\big((\theta_0 \otimes \id)\circ \tilde\Delta a_\eta + \un \otimes a_{x_0\eta}\big)\\
		& \quad + (\theta_1 \otimes \mu)\circ(\tilde{\Delta}\tilde\theta_{i}a_{\eta'}\otimes a_{\eta''}
			+ \tilde{\Delta}a_{\eta'}\otimes \tilde\theta_{i}a_{\eta''})\\		
		&=  (\tilde\theta_{i} \otimes \id + \id \otimes \tilde\theta_{i}
			+ \delta_{0,i} \tilde\theta_1 \otimes \kappa_\emptyset)
		\big((\theta_0 \otimes \id)\circ \tilde\Delta a_\eta + \un \otimes a_{x_0\eta}\big)\\
		& \quad +
		(\theta_1 \otimes \mu)\circ
		\Big( \big((\tilde\theta_{i} \otimes \id + \id \otimes \tilde\theta_{i}
		+ \delta_{0,i} \tilde\theta_1 \otimes \kappa_\emptyset)\circ  \Delta a_{\eta'}\big)\otimes a_{\eta''} \\
		&\quad	- \big(\un \otimes \tilde\theta_{i}a_{\eta'}\big)\otimes a_{\eta''}
				+ \tilde{\Delta}a_{\eta'}\otimes \tilde\theta_{i}a_{\eta''}\Big)\\
		&=  (\tilde\theta_{i} \otimes \id + \id \otimes \tilde\theta_{i}
				+ \delta_{0,i} \tilde\theta_1 \otimes \kappa_\emptyset)
		\Big((\theta_0 \otimes \id)\circ \tilde\Delta a_\eta + \un \otimes a_{x_0\eta}\Big)\\
		& \quad +
		(\theta_1 \otimes \mu)\circ
		\Big( \big((\tilde\theta_{i} \otimes \id
			+ \delta_{0,i} \tilde\theta_1 \otimes \kappa_\emptyset) \tilde\Delta a_{\eta'}\big)\otimes a_{\eta''} \\
		&\quad	+ \big((\id \otimes \tilde\theta_{i})\circ  \Delta a_{\eta'}\big)\otimes a_{\eta''}
				- \big( \un \otimes \tilde\theta_{i}a_{\eta'} \big)\otimes a_{\eta''}
				+ \tilde{\Delta}a_{\eta'}\otimes \tilde\theta_{i}a_{\eta''}\Big)\\	
		&=  (\tilde\theta_{i} \otimes \id + \id \otimes \tilde\theta_{i}+ \delta_{0,i} \tilde\theta_1 \otimes \kappa_\emptyset)
		\Big((\theta_0 \otimes \id)\tilde\Delta a_\eta + \un \otimes a_{x_0\eta}\Big)\\
		& \quad 	+ (\theta_1 \otimes \mu)\circ \Big( \big((\tilde\theta_{i} \otimes \id
				+ \delta_{0,i} \tilde\theta_1 \otimes \kappa_\emptyset) \tilde\Delta a_{\eta'}\big)\otimes a_{\eta''} \\
		&\quad	+ \big((\id \otimes \tilde\theta_{i}) \tilde\Delta a_{\eta'}\big)\otimes a_{\eta''}
				+ \tilde{\Delta}a_{\eta'}\otimes \tilde\theta_{i}a_{\eta''}\Big).	
		\end{align*}}%
Applying the Leibniz rule for the right-augmentation operator $\tilde\theta_{i}$ to the product $(\id \otimes \mu)(\tilde{\Delta}a_{\eta'}\otimes a_{\eta''})$  gives
\allowdisplaybreaks{
\begin{align*}
	\Delta a_{x_0 \eta x_i}
&=  (\tilde\theta_{i} \otimes \id + \id \otimes \tilde\theta_{i}+ \delta_{0,i} \tilde\theta_1 \otimes \kappa_\emptyset)
		\Big((\theta_0 \otimes \id)\tilde\Delta a_\eta + \un \otimes a_{x_0\eta}\Big) \\
		& \quad 	+ (\tilde\theta_{i} \otimes \id + \id \otimes \tilde\theta_{i}
				+ \delta_{0,i} \tilde\theta_1 \otimes \kappa_\emptyset) (\theta_1 \otimes \mu)
				\circ (\tilde{\Delta}\otimes \id)\circ\Delta_{\shuffle}a_{\eta}\\
		&=  (\tilde\theta_{i} \otimes \id + \id \otimes \tilde\theta_{i}+ \delta_{0,i} \tilde\theta_1 \otimes \kappa_\emptyset)
				\circ \Delta a_{x_0\eta}\\
		&= \tilde\Theta_{x_0 \eta x_i} \circ \Delta a_{\emptyset},%
\end{align*}}%
which proves the theorem.
\end{proof}

The new coproduct formula above provides a simple recursive formula for the antipode of the output feedback Hopf algebra $H$.
This is the main result of this section.

\begin{prop}\label{prop:antipodeSpecial}
For any nonempty word $\eta= x_{i_1} \cdots x_{i_l}$, the antipode $S: H \to H$ in Theorem~\ref{th:output-feedback-HA} can be written as
\begin{equation*}
	Sa_\eta =  (-1)^{|\eta|-1}\tilde\Theta'_{\eta}(a_\emptyset),
\end{equation*}
where $|\eta|=l$ is the length of the word $\eta$, and
$$
	 \tilde\Theta'_{\eta}:= \tilde\theta'_{{i_l}} \circ \cdots \circ  \tilde\theta'_{i_1}
$$
with
$$
	\tilde\theta'_{1}(a_\eta):=-a_{\eta x_1}
$$
and
$$
	\tilde\theta'_{0}(a_\eta):= -\tilde\theta_{0}(a_\eta) +  a_\emptyset\tilde\theta_{1}(a_\eta)
	=-a_{\eta x_0} + a_{\eta x_1}a_\emptyset.
$$
\end{prop}

\begin{proof}
The proof is via induction on the degree of $a_\eta$. The degree one case is excluded by assumption. For degree two and three it is quickly verified that
$$
	 \tilde\Theta'_{x_1}(a_\emptyset) = \tilde\theta'_{1}(a_\emptyset)=-a_{x_1} = Sa_{x_1},
$$
and
$$
	 \tilde\Theta'_{x_0}(a_\emptyset) = \tilde\theta'_{0}(a_\emptyset)=-a_{x_0}+a_{x_1}a_\emptyset = Sa_{x_0}.
$$
Now assume the theorem holds up to degree $n\geq 2$, and suppose $\deg(a_\eta)=n$. There are multiple cases to address. Using Sweedler's notation for the reduced coproduct
$$
	\Delta' a_\eta:= \Delta a_\eta - \un \otimes a_\eta - a_\eta \otimes \un =: {\sum}' a_{\eta(1)} \otimes a_{\eta(2)},
$$
first consider
the antipode $Sa_{\eta x_1} = - a_{\eta x_1} - {\sum}' (Sa_{\eta x_1(1)})a_{\eta x_1(2)}$. By definition of the coproduct,
 the letter $x_1$ appears always as the last letter in one of the factors in the sum ${\sum}' (Sa_{\eta x_1(1)})a_{\eta x_1(2)}$, which therefore splits as
$$
	{\sum}' (Sa_{\eta x_1(1)})a_{\eta x_1(2)} =
	{\sum_{\eta}}'  (Sa_{\eta(1)})a_{\eta(2) x_1} + {\sum_{\eta}}' (Sa_{\eta(1) x_1})a_{\eta(2)}.	
$$
Here the sum notation $ {\sum}'_{\eta}$ will indicate that the reduced coproduct of $\eta$ is being used.
Recall that in ${\sum}'_{\eta} (Sa_{\eta x_1(1)})a_{\eta x_1(2)}$, $a_{\eta x_1(1)} \in V$, whereas the term $a_{\eta x_1(2)} \in H$ may be a monomial.
Furthermore, ${\sum}'_{\eta} (Sa_{\eta(1)})a_{\eta (2) x_1}$ means that $x_1$ appears as the last letter in exactly one of the factors of this monomial.
Thus,
\allowdisplaybreaks{
\begin{align}
	Sa_{\eta x_1} &= \tilde\theta_{1}(- a_\eta)
					- {\sum_{\eta}}'  (Sa_{\eta(1)}) a_{\eta(2) x_1}
					- {\sum_{\eta}}'  (Sa_{\eta(1) x_1}) a_{\eta (2)} \label{simpleSTEP}\\
				&= \tilde\theta_{1}(- a_\eta)
					- {\sum_{\eta}}'  (Sa_{\eta(1)})  \tilde\theta_{1} (a_{\eta (2)})
					- {\sum_{\eta}}'   (S\tilde\theta_{1}(a_{\eta (1)}))a_{\eta (2)} \nonumber \\
				&= \tilde\theta'_{1}(a_\eta )
					+ \tilde\theta'_{1} {\sum_{\eta}}' (Sa_{\eta(1)}) a_{\eta (2)} \nonumber \\	
				&=  - \tilde\theta'_{1} (Sa_\eta) = - (-1)^{|\eta |-1} \tilde\theta'_{1}\tilde\Theta'_{\eta}(a_\emptyset)\nonumber\\
				&= (-1)^{|\eta x_1 |-1} \tilde\Theta'_{\eta x_1}(a_\emptyset). \nonumber %
\end{align}}%
Now the second case is considered, namely the antipode $S a_{\eta x_0} = - a_{\eta x_0} - {\sum}' (Sa_{\eta x_0(1)})a_{\eta x_0(2)}$. Analogous to the previous case, the suffix $x_0$ can be attached to the left or right factor in the term ${\sum}' (Sa_{\eta x_0(1)})a_{\eta x_0(2)}$. However, due to the particular nature of the letter $x_0$,  there is a third case, that is, when $x_0$ as the last letter in a word is {\em split} according to the coproduct $\Delta a_{x_0} = a_{x_0} \otimes \un + \un \otimes a_{x_0} + a_{x_1} \otimes a_{\emptyset}$. The calculation to show that $S a_{\eta x_0} =(-1)^{|\eta x_0 |-1} \tilde\Theta'_{\eta x_0}(a_\emptyset)$ goes as follows. First observe that
\allowdisplaybreaks{
\begin{align*}
	 (-1)^{|\eta x_0 |-1} \tilde\Theta'_{\eta x_0}(a_\emptyset)
	 &= (-1)^{|\eta x_0 |-1}  \tilde\theta'_{0}\tilde\Theta'_{\eta}(a_\emptyset)\\
	 &= -   \tilde\theta'_{0} \big((-1)^{|\eta|-1}\tilde\Theta'_{\eta}(a_\emptyset) \big)\\
	 &= -   \tilde\theta'_{0}(Sa_\eta).
\end{align*}}
Applying the Leibniz rule yields
\allowdisplaybreaks{
\begin{align*}
	 -   \tilde\theta'_{0}(S a_\eta)
	 &= -   \tilde\theta'_{0} \left(- a_{\eta} - {\sum_{\eta}}' (Sa_{\eta(1)})a_{\eta(2)}\right)\\
	 &=     \tilde\theta_{0} \left(- a_{\eta} - {\sum_{\eta}}' (Sa_{\eta(1)})a_{\eta(2)}\right)
	 	-  \kappa_\emptyset \tilde\theta_{1} \left(- a_{\eta}
		- {\sum_{\eta}}' (Sa_{\eta(1)})a_{\eta(2)}\right)\\\
	&= - a_{\eta x_0} + a_{\eta x_1}a_\emptyset
		- \tilde\theta_{0} \left({\sum_{\eta}}' (Sa_{\eta(1)})a_{\eta(2)}\right)
		+ \kappa_\emptyset \tilde\theta_{1} \left({\sum_{\eta}}' (Sa_{\eta(1)})a_{\eta(2)}\right)\\
	&= - a_{\eta x_0} + a_{\eta x_1}a_\emptyset \\
	& \quad	
		- {\sum_{\eta}}' \tilde\theta_{0} (Sa_{\eta(1)})a_{\eta(2)}
		- {\sum_{\eta}}' (Sa_{\eta(1)}) \tilde\theta_{0} (a_{\eta(2)})\\
	& \quad	
		+ {\sum_{\eta}}' \kappa_\emptyset \tilde\theta_{1} (Sa_{\eta(1)})a_{\eta(2)}
		+ {\sum_{\eta}}' (Sa_{\eta(1)}) \kappa_\emptyset \tilde\theta_{1} (a_{\eta(2)})\\
	&= - a_{\eta x_0}
		+ {\sum_{\eta}}'( -\tilde\theta_{0} +  \kappa_\emptyset \tilde\theta_{1} )(Sa_{\eta(1)})a_{\eta(2)}
		- {\sum_{\eta}}' (Sa_{\eta(1)}) \tilde\theta_{0} (a_{\eta(2)})\\
	& \quad	
		+ a_{\eta x_1}a_\emptyset
		+ {\sum_{\eta}}' (Sa_{\eta(1)}) \kappa_\emptyset \tilde\theta_{1} (a_{\eta(2)})\\
	&= - a_{\eta x_0}
		- {\sum_{\eta}}' (S\tilde\theta_{0}a_{\eta(1)})a_{\eta(2)}
		- {\sum_{\eta}}' (Sa_{\eta(1)}) \tilde\theta_{0} a_{\eta(2)}\\
	& \quad	
		+ a_{\eta x_1}a_\emptyset
		+ {\sum_{\eta}}' (Sa_{\eta(1)}) \tilde\theta_{1} (a_{\eta(2)})a_\emptyset.%
\end{align*}}%
Note in the final step above, the following identity was employed, which is a consequence of the
induction hypothesis. Namely,
\allowdisplaybreaks{
\begin{align*}
	{\sum_{\eta}}'( -\tilde\theta_{0} +  \kappa_\emptyset \tilde\theta_{1} )(Sa_{\eta(1)})a_{\eta(2)}
	&= {\sum_{\eta}}'\tilde\theta'_{0} (Sa_{\eta(1)})a_{\eta(2)}\\
	&= {\sum_{\eta}}'\tilde\theta'_{0} \big((-1)^{|\eta(1)|-1}\tilde\Theta'_{\eta(1)}(a_\emptyset)\big)a_{\eta(2)}\\
	&= - {\sum_{\eta}}'\big((-1)^{|\eta(1)x_0|-1}\tilde\theta'_{0} \tilde\Theta'_{\eta(1)}(a_\emptyset) \big)a_{\eta(2)}\\
	&= - {\sum_{\eta}}'(S\tilde\theta_{0}a_{\eta(1)})a_{\eta(2)}.%
\end{align*}}%
The calculation is completed via the observation that line \eqref{simpleSTEP} multiplied by $a_\emptyset$ implies that $- (Sa_{\eta x_1})a_\emptyset $ can be rearranged as
$$
	a_{\eta x_1}a_\emptyset
		+ {\sum_{\eta}}' (Sa_{\eta(1)})\tilde\theta_{1} (a_{\eta(2)})a_\emptyset
	= - (Sa_{\eta x_1})a_\emptyset
	- {\sum_{\eta}}' (S\tilde\theta_{1} a_{\eta(1)})a_{\eta(2)}a_\emptyset.
$$
Therefore,
\allowdisplaybreaks{
\begin{align*}
	-   \tilde\theta'_{0}(Sa_\eta)
	&= - a_{\eta x_0}
		- {\sum_{\eta}}'(S\tilde\theta_{0}a_{\eta(1)})a_{\eta(2)}
		- {\sum_{\eta}}' (Sa_{\eta(1)}) \tilde\theta_{0} a_{\eta(2)}\\
	& \quad	
		- (Sa_{\eta x_1})a_\emptyset - {\sum_{\eta}}' (S\tilde\theta_{1} a_{\eta(1)})a_{\eta(2)}a_\emptyset,
\end{align*}}%
which implies that $ (-1)^{|\eta x_0 |-1} \tilde\Theta'_{\eta x_0}(a_\emptyset)= S a_{\eta x_0}$. To see that the right-hand side of the last equality is equivalent to $S a_{\eta x_0}$, simply expand the latter using the standard antipode recursion defined in terms of the coproduct characterized in Lemma \ref{le:tilde-delta-inductions}.
\end{proof}

An elementary application of this theorem is the following corollary, which essentially states that
the coefficients of the antipode for coordinate functions ending in the letter $x_0$
always sum to zero. For example, $Sa_{x_0x_0}=-a_{x_0x_0}+a_{x_1}a_{x_0}+a_{x_1 x_0}a_\emptyset+ a_{x_0x_1}a_\emptyset-
a_{x_1}a_{x_1} a_{\emptyset}-a_{x_1x_1}a_{\emptyset}a_{\emptyset}$, so that $-1+1+1+1-1-1=0$. This claim is not so obvious
when approached by other
means.

\begin{cor}\label{co:Berlin-identity}
For any $\eta\in X^\ast$, $Sa_{\eta x_0}(\sum_{\xi\in X^\ast}\xi)=0$.
\end{cor}

\begin{proof}
Observe
\begin{align*}
	Sa_{\eta x_0}\left(\sum_{\xi\in X^\ast} \xi\right)
	&=-\tilde\theta'_{0}(Sa_\eta)\left(\sum_{\xi\in X^\ast} \xi\right) \\
	&= -(-\tilde\theta_{0}(a_\eta) +  a_\emptyset\tilde\theta_{1}(a_\eta))\left(\sum_{\xi\in X^\ast} \xi\right) \\
	&=(+a_{\eta x_0} - a_{\eta x_1}a_\emptyset)\left(\sum_{\xi\in X^\ast} \xi\right) \\
	&=+1-1=0.
\end{align*}
\end{proof}

Finally, it worth pointing out that the multivariable generalization of Theorem~\ref{prop:antipodeSpecial} is straightforward.
In this situation, each coordinate function is given a superscript to indicate which component of the generating
series $c\in\re^m\langle\langle X \rangle\rangle$ it is associated with, i.e.,
\begin{displaymath}
	a^j_\eta(c):=\langle c_j,\eta\rangle,
\end{displaymath}
where $j=1,2,\ldots,m$ and $\eta\in X^\ast$ with the alphabet $X:=\{x_0,x_1,\ldots,x_m\}$. In which case, the statement of the generalized proposition
is identical to the current version except for the maps
$$
	\tilde\theta'_{i}(a^j_\eta):=-a^j_{\eta x_i},\quad i=1,2,\ldots,m
$$
and
$$
	\tilde\theta'_{0}(a^j_\eta):=-a^j_{\eta x_0} + \sum_{i=1}^m a^j_{\eta x_i}a^i_\emptyset.
$$


\section{Devlin's recursion from the feedback antipode}
\label{sect:devlin-recursion}

In this section it is shown that Devlin's recursion (\ref{eq:Devlin-recursion}) follows from the solution of the Abel equation (\ref{eq:Devlineq-without-gamma})
when written in terms of the output feedback antipode, that is, by combining Theorem~\ref{thm:Devlin-polynomials} and Proposition~\ref{prop:antipodeSpecial}.
The main step in creating this link is the following theorem.

\begin{thm}
\label{thm:antipoderecursion1}
Let $c_\delta := \delta + \sum_{k\ge 0} k! x_1^k$ be Ferfera's series. Then the following identities hold for any word $\eta \in X^*$
\begin{equation*}
	Sa_{\eta x_1}( - c) = \deg(a_\eta) Sa_{\eta} (- c)
\end{equation*}
\begin{equation*}
	Sa_{\eta x_0}(- c) = \deg(a_\eta) Sa_{\eta} (- c).
\end{equation*}
\end{thm}

The following corollary is a Hopf algebraic formulation and proof of Devlin's recursion \eqref{eq:Devlin-recursion}.

\begin{cor}\label{cor:DevlinAntipodeRec}
For $n\geq 1$ and $c=\sum_{k\geq 0} k!\,x_1^k$, Devlin's polynomials $a_n=((-c)^{\circ -1})_n=\sum_{\eta\in X^*_n} Sa_\eta(-c)\eta$ satisfy
\begin{equation*}
	a_n=	 (n-1) \sum_{\eta\in X^*_{n-1}} Sa_\eta(-c)\eta x_1 +
		 (n-2) \sum_{\eta\in X^*_{n-2}} Sa_\eta(-c)\eta x_0.
\end{equation*}
\end{cor}

The fact that the words in Ferfera's series $c=\sum_{k\geq 0} k!\,x_1^k$ only contain the letter $x_1$ gives the following.

\begin{cor}\label{cor:DevlinAntipodeSimple}
For $c=\sum_{k\geq 0} k!\,x_1^k$
\begin{equation*}
	Sa_\eta(-c) = (-1)^{|\eta|-1}\widehat{\Theta}'_{\eta}(a_\emptyset)(-c),
\end{equation*}
where
$$
	\widehat{\Theta}'_{\eta} := \tilde\theta'_{\widehat{i_l}} \circ \cdots \circ  \tilde\theta'_{\widehat{i_1}}
$$
with $\tilde\theta'_{\widehat{1}}:=\tilde\theta'_{1}$ and $\tilde\theta'_{\widehat{0}}:= a_\emptyset\tilde\theta_{1}.$
\end{cor}

This last corollary implies that the calculation of the coefficient $Sa_\eta(-c)$ reduces to evaluating certain polynomials in only
the coordinate functions $a_{x_1^n}$ and $a_\emptyset$
against the series $-c$. For instance,
\allowdisplaybreaks{
\begin{align*}
	\widehat{\Theta}'_{x_1x_0} (a_\emptyset) &= a_{x_1x_1}a_\emptyset  \\
	\widehat{\Theta}'_{x_1x_1} (a_\emptyset) &= a_{x_1x_1} \\
	\widehat{\Theta}'_{x_0x_1} (a_\emptyset) &= a_{x_1x_1}a_\emptyset + a_{x_1}a_{x_1}\\
	\widehat{\Theta}'_{x_0x_0} (a_\emptyset) &= a_{x_1x_1}a_\emptyset a_\emptyset + a_{x_1}a_{x_1}a_\emptyset\\
	\widehat{\Theta}'_{x_0x_0x_1} (a_\emptyset)&=
		a_{x_1x_1x_1}a_\emptyset a_\emptyset
		+  4 a_{x_1x_1}a_{x_1}a_\emptyset
		+  a_{x_1}a_{x_1}a_{x_1}.
\end{align*}}%
So the coordinate functions involving the letter $x_0$ play no role in the calculation.

The proof of the main result is accomplished by considering the adjoint of the right-augmentation map.
Note that for
any word $\eta x_1 \in X^*$ it follows that $a_{\eta' x_1} (\eta x_1) = (\tilde\theta_1 a_{\eta'}) (\eta x_1) = a_{\eta'} (\tilde x_1^{-1} (\eta x_1)) = \delta_{\eta,\eta'}$. Therefore,
the adjoint of the right-augmentation map $\tilde\theta_1$ is $\tilde\theta^*_1:=\tilde x_1^{-1} $ so that $\theta^*_1 (\eta x_1) = \eta$ and zero otherwise.
A key observation is that the adjoint $\tilde\theta^*_1$ acts as a coderivation, i.e., for the feedback group element $c_\delta = \delta + c =\delta + \sum_{\eta\in X^\ast}\langle c,\eta \rangle\eta$ it satisfies
$$
	\langle \tilde\theta_1(a_{\eta}a_{\nu}) , c_\delta \rangle
	= \langle a_{\eta} \otimes a_{\nu} , \Delta^* \tilde\theta^*_1(c_\delta) \rangle
	= \langle a_{\eta} \otimes a_{\nu} , \tilde\theta^*_1 c_\delta \otimes c_\delta +  c_\delta \otimes \tilde\theta^*_1c_\delta \rangle,
$$
where $\Delta^*$ is the coproduct dual to the product $\mu$ in $H$, and $c_\delta$ is a group-like element, i.e., a character on $H$.

\begin{proof}[Proof of Theorem \ref{thm:antipoderecursion1}]
The proof is by induction on the degree of the coordinate functions.
So the claim is first verified for $a_{x_1x_1}$ and  $a_{x_0x_1}$, which have degrees three and four, respectively.
First observe that if $(-c)_\delta :=  \delta + (-c) =  \delta + \sum_{k\ge 0} (-k!) x_1^k$, then
$\tilde\theta^*_1 c_\delta = \un +\sum_{k>0} (k+1)!\, x_1^k$. Therefore,
applying Proposition~\ref{prop:antipodeSpecial}, as well as the general identity $S a_{\eta} = - a_{\eta x_1} - {\sum}' (Sa_{\eta(1)})a_{\eta(2)}$
(again Sweedler's notation for the reduced coproduct is used, namely, $\Delta'(a_\eta) =  {\sum}' a_{\eta(1)} \otimes a_{\eta(2)}$) gives the following:
\begin{align*}
	Sa_{x_1x_1} ( - c )
	&= - \tilde\theta'_1 (Sa_{x_1}) ( - c ) \\	
	&= \tilde\theta_1 (Sa_{x_1}) (- c ) \\	
	&= Sa_{x_1} (\tilde\theta^*_1 (-c)) \\
	&= - a_{x_1} \left(\sum_{k\ge 0} -(k+1)!\, x_1^k \right) \\
	&= 2 Sa_{x_1} (- c). 		
\end{align*}
For the coordinate function $a_{x_0x_1}$:
\begin{align*}
	Sa_{x_0x_1} (- c)
	&= - \tilde\theta'_1 (Sa_{x_0}) (- c) \\	
	&=  \tilde\theta_1 (Sa_{x_0}) ( - c) \\	
	&= Sa_{x_0} ( \tilde\theta^*_1 (-c)) \\
	&= (a_{x_1}a_\emptyset)( \tilde\theta^*_1 (-c)) \\
	&= \langle a_{x_1} \otimes a_\emptyset , \tilde\theta^*_1 c_\delta \otimes c_\delta
									+  c_\delta \otimes \tilde\theta^*_1c_\delta \rangle \\
	&=3 Sa_{x_0} (- c ). 		
\end{align*}
Now assume the identity holds up to some fixed degree $n\geq 4$ and select $\eta$ such that $\deg(\eta x_1)=n+1$ and $|\eta|_{x_0} \neq 0$.
Then, using
the fact that $\deg(a_{\eta(1)})=\deg(Sa_{\eta(1)})$, observe
\begin{align*}
	 Sa_{\eta x_1} ( - c)
	&= \Big(- a_{\eta x_1} - {\sum}' (Sa_{\eta x_1(1)})a_{\eta x_1(2)}\Big)( - c)\\
	&= \Big( - {\sum}' (Sa_{\eta x_1(1)})a_{\eta x_1(2)}\Big)( - c)\\
	&= \Big(-  \tilde\theta'_1{\sum}' (Sa_{\eta(1)})a_{\eta (2)}\Big)( - c)\\
	&= \left\langle {\sum}' (Sa_{\eta(1)})a_{\eta (2)} , \tilde\theta^*_1 c_\delta \otimes c_\delta
						+  c_\delta \otimes \tilde\theta^*_1c_\delta \right\rangle\\
	&= \Big({\sum}' \deg(a_{\eta(1)}) (Sa_{\eta(1)})a_{\eta (2)}\Big)( -c)
						+ \Big({\sum}'  \deg(a_{\eta(2)}) (Sa_{\eta(1)})a_{\eta (2)}\Big)( -c)\\
	&={\sum}' \big( \deg(a_{\eta(1)}) +  \deg(a_{\eta(2)}) \big) ((Sa_{\eta(1)})a_{\eta (2)}) (-c )\\
	&=  \deg(a_{\eta}) Sa_{\eta} (-c).							
\end{align*}%
The case where $\eta = x_1^n$ reduces directly to $Sa_\eta = - a_\eta$, and therefore
\begin{align*}
	Sa_{\eta x_1} ( - c )
	&= - \tilde\theta'_1 (Sa_{\eta}) ( - c ) \\	
	&= \tilde\theta_1 (Sa_{\eta}) (- c ) \\	
	&= Sa_{\eta} (\tilde\theta^*_1 (-c)) \\
	&= - a_{\eta} \left(\sum_{k\ge 0} -(k+1)!\, x_1^k \right) \\
	&= (n+1) Sa_{\eta} (- c). 		
\end{align*}

The theorem is proved.
\end{proof}




\begin{thebibliography}{99}

\bibitem{Alwash-LLoyd_87}
	M.~A.~M.~Alwash, N.~G.~Lloyd,
	{\textit{Nonautonomous equations related to polynomial two-dimensional systems}},
	Proc. Roy. Soc. Edinburgh {\bf{105A}} (1987) 129--152.

\bibitem{Briskin-etal_10}
	M.~Briskin, N.~Roytvarf, Y.~Yomdin,
	{\em Center conditions at infinity for Abel differential equation},
	Ann.~of Math.(2) {\bf 172} (2010) 437--483.

\bibitem{Brudnyi_10}
A.~Brudnyi,
{\em Some algebraic aspects of the center problem for ordinary differential equations},
Qual. Theory Dyn. Syst. {\bf 9} (2010) 9--28.

\bibitem{cartier1}
        P.~Cartier,
        {\it{A primer of Hopf algebras}},
        in ``Frontiers in Number Theory, Physics, and Geometry II'',
        Springer, Berlin Heidelberg, 2007, pp.~537--615.

\bibitem{Cherkas_76}
L.~Cherkas, {\em Number of limit cycles of an autonomous second-order system}, 
Differ.~Uravn. {\bf 12} (1976) 944--946.

\bibitem{ck}
	A.~Connes, D.~Kreimer,
	{\textit{Hopf algebras, renormalization and noncommutative geometry}},
	Commun.~Math.~Phys. {\bf{199}} (1998) 203--242.

\bibitem{Devlin_1989}
	J.~Devlin,
	{\textit{Word problems related to periodic solutions of a nonautonomous system}},
	Math. Proc. Cambridge Philos. Soc. {\bf{108}} (1990) 127--151.

\bibitem{Devlin_1991}
	J.~Devlin,
	{\textit{Word problems related to derivatives of the displacement map}},
	Math. Proc. Cambridge Philos. Soc. {\bf{110}} (1991) 569--579.

\bibitem{DuffautEFG_2014}
	L. A. Duffaut Espinosa, K. Ebrahimi-Fard, W. S. Gray,
	A combinatorial Hopf algebra for nonlinear output feedback control systems,	
	{\tt http://lanl.arxiv.org/abs/1406.5396}.

\bibitem{Ferfera_79}
	A.~Ferfera,
	Combinatoire du Mono\"{i}de Libre Appliqu\'{e}e \`{a} la Composition et aux Variations de Certaines Fonctionnelles
	Issues de la Th\'{e}orie des Syst\`{e}mes, Doctoral dissertation, University of Bordeaux I, 1979.

\bibitem{Ferfera_80}
	A.~Ferfera,
	{\textit{Combinatoire du mono\"{i}de libre et composition de certains syst\`{e}mes non lin\'{e}aires}},
	Ast\'{e}risque~{\bf{75}}--{\bf{76}} (1980) 87--93.

\bibitem{Figueroa-Gracia-Bondia_05}
	H.~Figueroa and J.~M.~Gracia-Bond\'{ii}a,
	{\textit{Combinatorial Hopf algebras in quantum field theory I}},
	Rev.~Math.~Phys.~{\bf{17}} (2005) 881--976.

\bibitem{Fliess_81}
	 M.~Fliess,
	 {\textit{Fonctionnelles causales non lin\'{e}aires et ind\'{e}termin\'{e}es non commutatives}},
	 Bull. Soc. Math. France~{\bf{109}} (1981) 3--40.

\bibitem{Fliess_83}
	M.~Fliess,
	{\textit{R\'{e}alisation locale des syst\`{e}mes non lin\'{e}aires, alg\`{e}bres
	de Lie filtr\'{e}es transitives et s\'{e}ries g\'{e}n\'{e}ratrices non commutatives}},
	Invent.~Math.~{\bf{71}} (1983) 521--537.

\bibitem{Foissy_13}
	L.~Foissy,
	{\it{The Hopf algebra of Fliess operators and its dual pre-Lie algebra}},
	Communications in Algebra {\bf{43}} (2015)  4528--4552.

\bibitem{manchonfrabetti}
	A. Frabetti, D. Manchon,
	{\textit{Five interpretations of Fa\`a di Bruno's formula}},
	in ``Fa\`{a} di Bruno Hopf Algebras, Dyson-Schwinger Equations, and Lie-Butcher Series",
	K.~Ebrahimi-Fard and F.~Fauvet, Eds.,
	IRMA Lect.~Math.~Theor.~Phys. {\bf{21}}, Eur.~Math.~Soc., Strasbourg, France, 2015, pp.~91--147.
	
\bibitem{GBFV}
	J.~M.~Gracia-Bond\'{i}a, J.~C.~V\'arilly, H.~Figueroa,
	Elements of Noncommutative Geometry,
	Birkh\"auser Boston, Boston, MA, 2001.

\bibitem{Gray-Duffaut_Espinosa_SCL11}
	W.~S.~Gray, L.~A.~Duffaut Espinosa,
	{\textit{A Fa\`{a} di Bruno Hopf algebra for a group of Fliess operators with applications to feedback}},
	Systems Control Lett.~{\bf{60}} (2011) 441--449.

\bibitem{Gray-Duffaut_Espinosa_FdB14}
	W.~S.~Gray, L.~A.~Duffaut Espinosa,
	{\textit{A Fa\`{a} di Bruno Hopf algebra for analytic nonlinear feedback control systems}},
	in ``Fa\`{a} di Bruno Hopf Algebras, Dyson-Schwinger Equations, and Lie-Butcher Series",
	K.~Ebrahimi-Fard and F.~Fauvet, Eds.,
	IRMA Lect.~Math.~Theor.~Phys. {\bf{21}}, Eur.~Math.~Soc., Strasbourg, France, 2015, pp.~149--217.

\bibitem{Gray-et-al_MTNS14}
	W.~S.~Gray, L.~A.~Duffaut Espinosa, K.~Ebrahimi-Fard,
	{\textit{Recursive algorithm for the antipode in the SISO feedback product}},
	Proc.\ 21$^{st}$ International Symposium on the Mathematical Theory
	of Networks and Systems, Groningen, The Netherlands, 2014, pp.~1088--1093.

\bibitem{Gray-et-al_SCL14}
	W.~S.~Gray, L.~A.~Duffaut Espinosa, K.~Ebrahimi-Fard,
	{\textit{Fa\`{a} di Bruno Hopf algebra of the output feedback group for multivariable Fliess operators}},
	Systems Control Lett.~{\bf{74}} (2014) 64--73.

\bibitem{Gray-Li_05}
	W.~S.~Gray, Y.~Li,
	{\textit{Generating series for interconnected analytic nonlinear systems}},
	SIAM J.~Control Optim.~{\bf{44}} (2005) 646--672.

\bibitem{Gray-Wang_SCL02}
	W.~S. Gray, Y.~Wang,
	{\textit{Fliess operators on $L_p$ spaces: Convergence and continuity}},
	Systems Control Lett.~{\bf{46}} (2002) 67--74.

\bibitem{Ilyashenko_02}
	Y.~Ilyashenko,
	{\em Centennial history of Hilbert's 16th problem},
	Bull.~AMS {\bf 39} (2002) 301--356.

\bibitem{Ilyashenko-Yakovenko_95}
	Y.~Ilyashenko, S.~Yakovenko,
	Eds., Concerning the Hilbert's 16th Problem,
	American Mathematical Society, Providence, RI, 1995.

\bibitem{Isidori_95}
	A.~Isidori,
	Nonlinear Control Systems, 3rd edition,
	Springer-Verlag, London, 1995.
	
\bibitem{Lloyd_82}
	N.~G.~Lloyd,
	{\it{Small amplitude limit cycles of polynomial differential equations}},
	in ``Ordinary Differential Equations and Operators'', W.~N.~Everitt and R.~T.~Lewis, Eds.,
	Lecture Notes in Mathematics {\bf{1032}}, Springer, Berlin, 1982, pp.~346--357.

\bibitem{manchon2}
	D.~Manchon,
	{\textsl{Hopf algebras and renormalisation}}, in
	``Handbook of Algebra'', {\bf 5}, M. Hazewinkel, Ed., Elsevier, Amsterdam, 2008, pp.~365--427.

\bibitem{Poincare_1892}
	H.~Poincar\'{e},
	{\em Sur les courbes d\'{e}finies par une \'{e}quation diff\'{e}rentielle}, Oeuvres,
	t.1, Gauthier-Villars et Cie, Paris, 1928.

\bibitem{reutenauer}
	C.~Reutenauer,
	Free Lie algebras,
	Oxford University Press, New York, 1993.
	
\bibitem{Sweedler_69}
	M.~E.~Sweedler,
	Hopf Algebras,
	W.~A.~Benjamin, Inc., New York, 1969.

\bibitem{Thitsa-Gray_12}
M.~Thitsa and W.~S.~Gray,
{\em On the radius of convergence of interconnected analytic nonlinear input-output systems},
SIAM J.~Control Optim.~{\bf 50} (2012) 2786--2813.

\bibitem{Wang_90}
	Y.~Wang,
	Algebraic Differential Equations and Nonlinear Control Systems,
Ph.D.~dissertation, Rutgers University, 1990.

\end{thebibliography}
\end{document}